\documentclass[12pt]{amsart}

\usepackage{graphicx}
\usepackage{amscd}
\usepackage[top=1in,bottom=1in,left=1.15in,right=1.15in]{geometry}
\usepackage{amsmath,amssymb,amsfonts}
\usepackage{amsthm}
\usepackage{color}
\usepackage{hyperref}
\usepackage{tikz-cd}
\usepackage{epstopdf}
\usetikzlibrary{matrix}
\usepackage{verbatim}
\usepackage{array}
\usepackage{mathtools}
\usepackage{enumitem}

\usepackage{amsmath,amssymb,amsfonts}
\usepackage{amsthm}
\usepackage{color}
\usepackage{hyperref}
\usepackage{tikz-cd}
\usepackage{epstopdf}
\usetikzlibrary{matrix}
\usepackage{verbatim}
\usepackage{array}
\usepackage{mathtools}
\usepackage{enumitem}

\newtheorem{theorem}{Theorem}[section]

\newtheorem{lemma}[theorem]{Lemma}
\newtheorem{proposition}[theorem]{Proposition}
\newtheorem{conjecture}[theorem]{Conjecture}

\theoremstyle{definition}
\newtheorem{definition}[theorem]{Definition}
\newtheorem{example}[theorem]{Example}

\newtheorem{corollary}[theorem]{Corollary}
\theoremstyle{remark}
\newtheorem{remark}[theorem]{Remark}

\allowdisplaybreaks

\def\diaCircle{\unitlength.1em
  \begin{minipage}{15\unitlength}
    \begin{picture}(15,15)
      \put(7.5,10){\circle{8}}
      \put(7.5,0){\vector(0,1){5}}
       \put(7.5,5){\line(0,1){5.5}}
    \end{picture}
  \end{minipage}
}

\usetikzlibrary{decorations.markings}

\tikzset{->-/.style={decoration={
  markings,
  mark=at position .5 with {\arrow{>}}},postaction={decorate}}}

\tikzset{-<-/.style={decoration={
  markings,
  mark=at position .5 with {\arrow{<}}},postaction={decorate}}}

\begin{document}
\title{Clock moves and Alexander polynomial of plane graphs}
\author{Wenbo Liao \and Zhongtao Wu}

\address{Department of Mathematics, The Chinese University of Hong Kong, Shatin, Hong Kong
}
\email{wbliao@math.cuhk.edu.hk \\ ztwu@math.cuhk.edu.hk }

\begin{abstract}
In this paper, we introduce a notion of clock moves for spanning trees in plane graphs. This enables us to develop a spanning tree model of an Alexander polynomial for a plane graph and prove the unimodal property of its associate coefficient sequence.  In particular, this confirms the trapezoidal conjecture for planar singular knots and gives new insights to Fox's original conjecture on alternating knots.   

\end{abstract}

\maketitle

\section{Introduction}

In \cite{BW}, Bao and the second author defined an Alexander polynomial $\Delta_{G}(t)$ for a special class of graphs, the \textit{MOY graphs}, which are directed spatial graphs with \textit{transverse orientation} and \textit{balanced coloring}.  The Alexander polynomial $\Delta_{G}(t)$ contains a wealth of topological as well as graph-theoretical information about the graph $G$. This paper further examines this polynomial invariant, primarily through the use of combinatorial methods.

In the case where $G$ is a plane MOY graph, the Alexander polynomial $\Delta_{G}(t)$ exhibits a number of particularly remarkable properties. For example, it is shown in \cite[Theorem 5.3]{BW} that $\Delta_G(t) \in \mathbb{Z}_{\geq0}[t]$, i.e., all non-zero coefficients of the Alexander polynomials are positive.  However, this result did not exclude the possibility of a zero coefficient in some intermediate-degree terms. The first result of this paper will show that such gaps can be excluded, thus establishing the  following ``strict positive'' property for plane graphs.

\begin{theorem} \label{Thm:strictpositivity}
Suppose $G$ is a plane graph and $\Delta_G(t)\doteq \sum_{i=1}^n a_it^i$ with $a_1=a_n\neq 0$.  Then all the coefficients $a_i$'s are positive.  Here and henceforth, the notation $\doteq$ denotes an equality up to a power of $t$. \end{theorem}

We note that the above theorem can be viewed as a graph analogue of the classical result in knot theory, which states that the Alexander polynomial of \textit{alternating knots} have non-zero coefficients of alternating signs \cite{MR0099666}. Indeed, additional restrictions on the Alexander polynomial of alternating knots have been conjectured by Fox \cite{Fox}, among which there is the famous \textit{Trapezoidal Conjecture} claiming that the coefficients of the Alexander polynomial of alternating knots exhibit a trapezoidal pattern. See \cite{AJK, HMV, Sto} for recent progress and generalizations of the conjecture. In light of the above Fox's conjecture, we proceed to verify the analogous trapezoidal conjecture for our Alexander polynomial of plane graphs.  

\begin{definition}
    Let $(a_1,a_2,\dots,a_n)$ be a sequence of non-negative integers such that $a_i=a_{n+1-i}$ for all $1\leq i\leq n$, and let $m=\lfloor \frac{n}{2}\rfloor$. We say that the sequence satisfies the \textit{trapezoidal} property if the following two conditions are met:
    \begin{enumerate}
        \item $a_1\leq a_2\cdots \leq a_m\geq\cdots\geq a_{n-1}\geq a_n$.

        \item If $a_i=a_{i+1}$ for some $1\leq i\leq m$, then $a_i=a_{i+1}=\cdots =a_m.$
    \end{enumerate}

   \noindent A sequence that satisfies only Property (1) is said to be \textit{unimodal}. 
\end{definition}

\begin{theorem}\label{Thm:Trapezoidal}
Suppose $G$ is a plane MOY graph and $\Delta_G(t)\doteq \sum_{i=1}^n a_i t^i$.  Then the sequence $(a_1,\dots,a_n)$ is unimodal.  
    
\end{theorem}

As an immediate corollary, note that a planar singular knot $\mathcal{K}$ can be viewed as a specific class of plane MOY graph with a trivial coloring. The Alexander polynomial $\Delta_\mathcal{K}(t)$ of singular knots has been previously defined in the literature via skein relations or state sum, cf. Ozsv\'ath-Stipsicz-Szab\'o \cite[Section 3]{OSS}.  It can be shown that this coincides with our more general version of the Alexander polynomial for a graph. Consequently,

\begin{corollary}
Suppose $\mathcal{K}$ is a planar singular knot and $\Delta_\mathcal{K}(t)\doteq \sum_{i=1}^n a_i t^i$.  Then the sequence $(a_1,\dots,a_n)$ is unimodal.
\end{corollary}

The main theory used to obtain the above results is the development of a \textit{spanning tree model} for the Alexander polynomial. The idea is to associate a power $t^{|T|}$ to each spanning tree $T$ of $G$, where $|T|$ denotes a suitable norm that can be defined, and the sum over all spanning trees yields the Alexander polynomial $\Delta_G(t)$ (Theorem \ref{Thm:spanningtreemodel}).  In light of this spanning tree model, we may reinterpret the trapezoidal property as a statistical assertion that \textit{there are a greater number of spanning trees with average norms}, a proposition that appears both intuitive and reasonable.    

\bigskip
The remainder of this paper is organized as follows.  In Section 2, we review the basic construction of state sums and Alexander polynomials. We give a simple proof of their invariance under the parallel edge replacement moves, which will help simplify our subsequent discussion of spanning tree models.  In Section 3, we introduce the notion of a clock move on spanning trees, which is inspired by Kauffman's classical definition of a clock move on Kauffman states in knot theory \cite{MR712133}. Likewise, we prove a clock theorem connecting all spanning trees by clock moves. In Section 4, we describe a canonical combinatorial bijection between Kauffman states and spanning trees for plane graphs. This, together with the clock theorem, enables us to develop a spanning tree model of the Alexander polynomial. In Section 5, we undertake a more detailed examination of the clock moves on spanning trees. It becomes evident that the planar condition of graphs provides a natural classification of clock moves into two categories: local and global. This is a key ingredient that is absent from Kauffman's classical theory and enables us to prove the trapezoidal conjecture in our context. In Section 6, we compare our spanning tree model with Crowell's model on alternating knots, hoping to give new insights to Fox's conjecture.  

\medskip\noindent
\textbf{Acknowledgements.}  The authors would like to thank Zhiyun Cheng for helpful discussions and suggestions.  The second author was partially supported by a grant from
the Hong Kong Research Grants Council (Project No. 14301819) and direct grants from CUHK.

\section{Background and Parallel edge replacements}\label{section:background}

In this section, we review the basic construction and properties of the Alexander polynomial and prove its invariance under the operation of parallel edge replacements.  For our purposes, we will not specify a framing of any graphs, so the Alexander polynomial is well-defined up to a power of $t$. 

\begin{definition} \label{Def:MOY}
    An \emph{MOY graph} $(G,c)$ is a transverse directed spatial graph $G$ with a \emph{balanced coloring} $c:E\to \mathbb{Z}_{\geq0}$, meaning:
    $$\sum_{e:\text{ pointing into }v}c(e)=\sum_{e:\text{ pointing out of }v}c(e).$$  In the special case where the balanced coloring $c$ is trivial, i.e., $c(e)=1$ for all edges $e$, the graph will be called a \textit{balanced transverse graph}. 
\end{definition}

    For a detailed and rigorous definition of a transverse directed spatial graph, see \cite{HO}. Roughly speaking, the transverse condition requires that at each vertex $v$, all the edges pointing into $v$ are on one side of $v$ while all the edges pointing out of $v$ are on the other side.  In the case of plane graphs, this implies a clockwise ordering of edges, with edges pointing into $v$ preceding those pointing out of $v$.  See Figure \ref{fig:transverse}.

    \begin{figure}[!h]
    \centering
    \includegraphics[width=10cm]{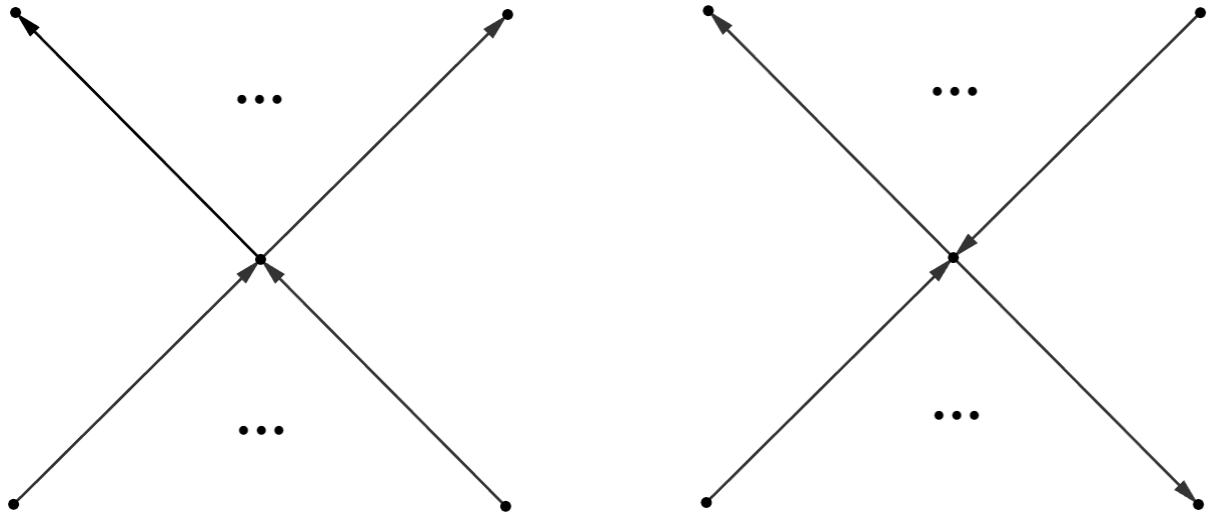}
    \caption{An orientation satisfying (left) / not satisfying (right) the transverse condition.}
    \label{fig:transverse}
\end{figure}

Following the exposition of \cite[Section 3]{BWSpanning}, we now describe the Kauffman states for a plane graph $G$. Starting from a plane graph $G$, we can obtain a {\it decorated diagram} $(G, \delta)$ by putting a base point $\delta$ on one of the edges in $G$ and drawing a circle around each vertex of $G$.  Define:\\

\begin{enumerate}

\item  $\operatorname{Cr}(G)$: denotes the set of crossings \diaCircle which are the intersection points around each vertex between the incoming edges with the circle.  Such a crossing is said to be generated by the edge. Denote by $cr(e)$ the crossing generated by edge $e$.\\
\item $\operatorname{Re}(G)$: denotes the set of regions, including the {\it regular regions} of $\mathbb{R}^{2}$ separated by $G$ and the {\it circle regions} around the vertices. Note that there is exactly one circle region around each vertex. {\it Marked regions} are the regions adjacent to the base point $\delta$, and the others are called {\it unmarked regions}.\\

\item Corners: There are three corners around a crossing \diaCircle, and we call the one inside the circle region the {\it north} corner, the one on the left of the crossing the {\it west} corner and the one on the right the {\it east} corner.  Note also that every corner belongs to a unique region in $\operatorname{Re}(G)$.  \\

\begin{figure}[h!]
\begin{tikzpicture}[baseline=-0.65ex, thick, scale=0.9]
\draw (0, 0.5) ellipse (1.5cm and 0.8cm);
\draw (0,-1) [->-] to (0,-0.3);
\draw (-0.4, -0.6) node {W};
\draw (0.4, -0.6) node {E};
\draw (0, 0) node {N};
\end{tikzpicture}
\end{figure}

\end{enumerate}

Calculating the Euler characteristic of $\mathbb{R}^2$ using $G$ shows
$$\vert \operatorname{Re}(G) \vert = \vert \operatorname{Cr}(G) \vert+2.$$
Also, a generic base point $\delta$ is adjacent to two distinct regions, which will be denoted by $R_u$ and $R_{v}$. 

\begin{definition}
\label{states}
A {\it Kauffman state} for a decorated diagram $(G, \delta)$ is a bijective map $$s: \, \operatorname{Cr}(G)\rightarrow \operatorname{Re}(G)\backslash \{R_u, R_{v}\},$$ which sends a crossing in $\operatorname{Cr}(G)$ to one of its corners. Denote $S(G, \delta)$ the set of all Kauffman states.
\end{definition}

\begin{definition}\label{Def:KauffmanStateSum}
Suppose $(G, \delta)$ is a decorated plane diagram with $n$ crossings $C_1, C_2, \cdots, C_{n} \in \operatorname{Cr}(G)$ and $n+2$ regions $R_1, R_2, \cdots, R_{n+2} \in \operatorname{Re}(G)$, and the base point $\delta$ is placed on an edge $e_0$ with color $i_0$. 

\begin{enumerate}
\item We define a local contribution $P_{C_p}^{\triangle}(t)$ as in Figure \ref{fig:f2}.
\begin{figure}[h!]
\begin{tikzpicture}[baseline=-0.65ex, thick, scale=0.9]
\draw (0, 0.5) ellipse (1.5cm and 0.8cm);
\draw (0,-2) [->-] to (0,-0.3);
\draw (-0.5, -0.6) node {$t^{-i/2}$};
\draw (0.5, -0.6) node {$t^{i/2}$};
\draw (0, 0.2) node {$[i]$};
\draw (0.3, -1.5) node {$i$};
\end{tikzpicture}
\hspace{2cm}
\begin{tikzpicture}[baseline=-0.65ex, thick, scale=0.9]
\draw (0, 0.5) ellipse (1.5cm and 0.8cm);
\draw (0,-2) [->-] to (0,-0.3);
\draw (0, 0.2) node {$t^{i_0/2}$};
\draw (0.4, -1.3) node {$i_0$};
\draw (0, -1.7) node {$*$};
\draw (-0.3, -1.7) node {$\delta$};
\end{tikzpicture}
	\caption{The local contributions $P_{C_p}^{\triangle}(t)$ for crossings generated by generic unmarked edges (left) and the marked edge (right), respectively.}
	\label{fig:f2}
\end{figure}
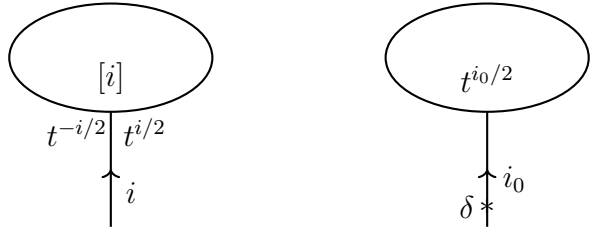

Here, $$[i]:= \frac{t^{i/2}-t^{-i/2}}{t^{1/2}-t^{-1/2}}=t^{\frac{i-1}{2}}+\cdots +t^{\frac{1-i}{2}}. $$

\item For each Kauffman state $s$, let $$P_s(t):= \prod_{p=1}^{n} P_{C_p}^{s(C_p)}(t).$$

\item The {\it Kauffman state sum/Alexander polynomial} is defined as
\begin{equation*} \label{equation: plane}
\Delta_{(G, c)}(t):=\sum_{s\in S(G, \delta)}  P_s(t).
\end{equation*}
This is well-defined up to a power of $t$ and is independent of the choice of $\delta$.  For further details, please refer to \cite[Section 2]{BW}. 

\end{enumerate}
\end{definition}

\bigskip

Next, we show the invariance of the Alexander polynomial under the move of \textit{parallel edge replacements}.

\begin{definition}\label{Def:paralleledge}
When we replace an edge $e$ of color $n$ in a plane MOY graph $G$ by $n$ parallel edges of color $1$, as depicted in Figure \ref{fig:parallel}, such an operation is called a parallel edge replacement.
\end{definition}

\begin{figure}[!h]
    \centering
    \includegraphics[width=10.5cm]{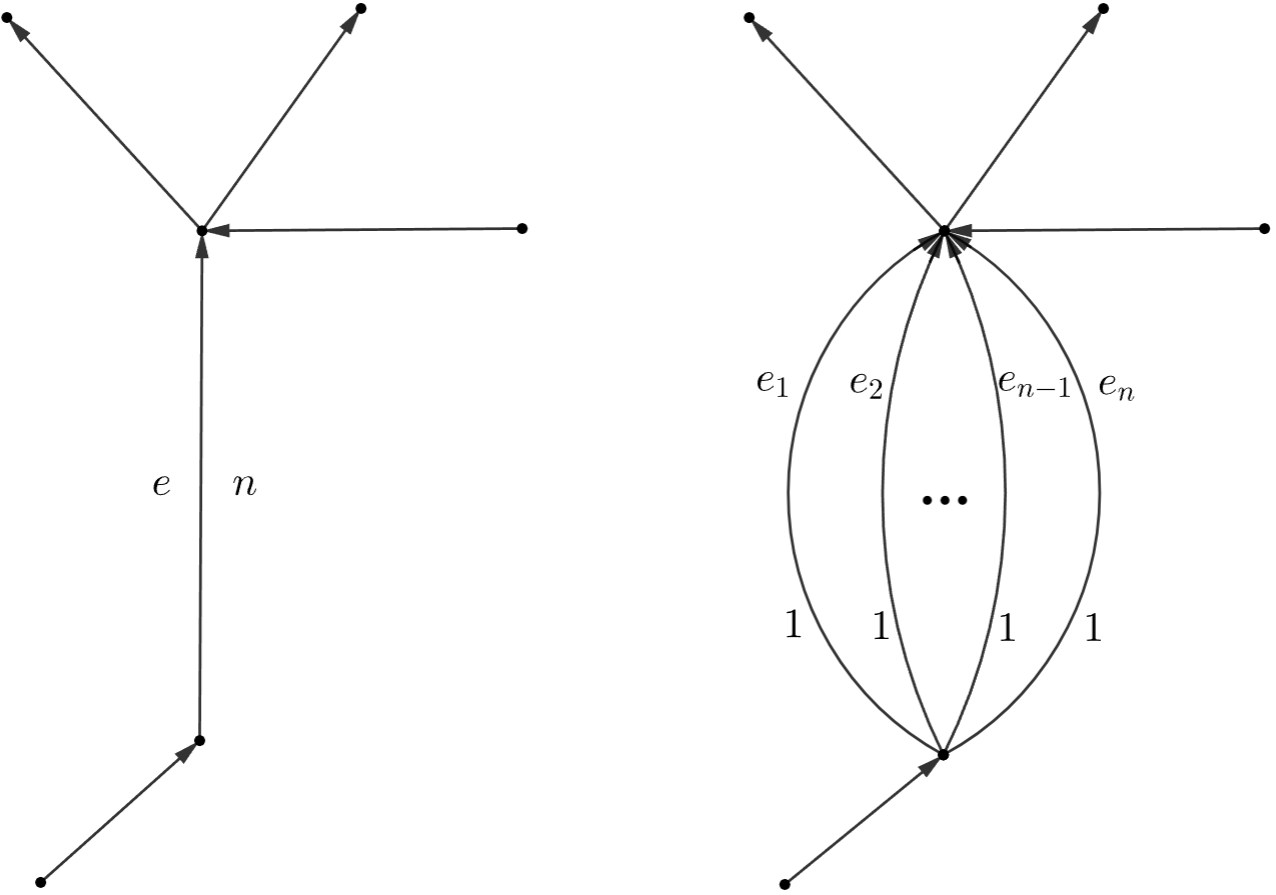}
    \caption{A parallel replacement move on an edge of color $n$}
    \label{fig:parallel}
\end{figure}

\begin{theorem}  \label{Thm:parallelinvariant}

If two plane MOY graphs $G$ and $G'$ are related by a sequence of parallel edge replacement operations, then their Alexander polynomials  $\Delta_{G}(t)= \Delta_{G'}(t)$. 

\end{theorem}

\begin{proof}

Suppose $e$ is an edge of color $n$ pointing at $v$ in a plane MOY graph $G$. We replace $e$ by $n$ parallel edges $e_1,\dots e_n$ pointing at $v$, and obtain a new plane graph $G'$.  Place the base point $\delta$ on $e$ for $G$ and on $e_1$ for $G'$, respectively. Note that each Kauffman state $s\in S(G,\delta)$ assigns $cr(e)$ to its north corner, while each Kauffman state $s' \in S(G',\delta)$ assigns $cr(e_1)$ to its north corner and $cr(e_2), \cdots, cr(e_n)$ to their respective east corners.  See Figure \ref{fig:replacement}.
This sets up an obvious bijection between $S(G, \delta)$ and $S(G', \delta)$ for two states $s$ and $s'$ that have identical corner assignments on all other crossings.  

\begin{figure}[!h]
    \centering
    \includegraphics[width=9cm]{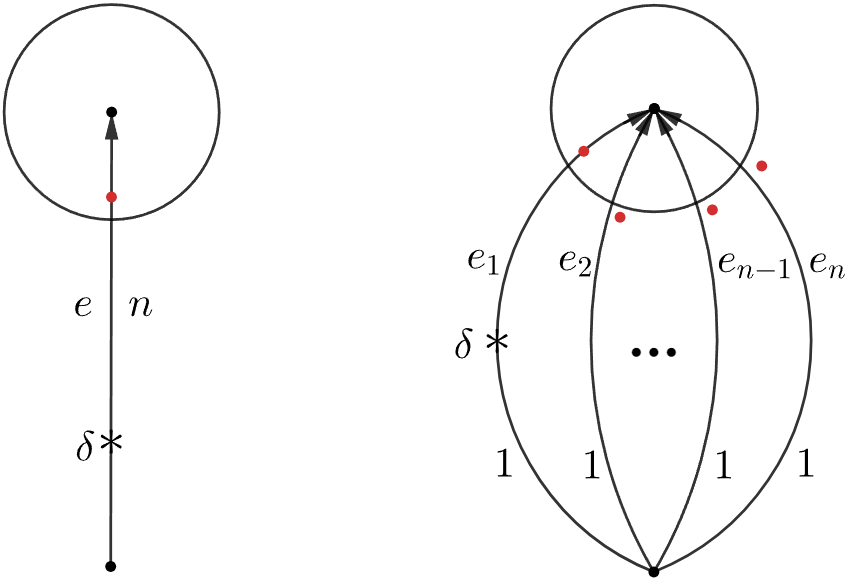}
    \caption{The corresponding Kauffman states resulting from a parallel replacement move on an edge of color $n$.}
    \label{fig:replacement}
\end{figure}

Using Definition \ref{Def:KauffmanStateSum}, we see that the local contribution of $s$ at $cr(e)$ is $t^{\frac{n}{2}}$, while the local contribution of $s'$ at each of the $n$ crossings $cr(e_i)$'s is $t^{\frac{1}{2}}$ whose product is also $t^{\frac{n}{2}}$. Since the assignments are the same on other crossings, the total contribution $P_s(t)$ of the Kauffman state $s$ on $G$ equals the total contribution $P_{s'}(t)$ of the corresponding Kauffman state $s'$ on $G'$. Summing over all Kauffman states gives the identity on Alexander polynomials.  

\end{proof}

Hence, for each plane MOY graph $G$, there exists a (canonical) plane balanced transverse graph $G'$ that may be constructed by replacing all edges of color $n>1$ in $G$ by $n$ parallel edges, which has the same Alexander polynomials. This reduction to trivial coloring is useful for simplifying multiple arguments and expositions in later sections of the paper.  In particular, note that in $G'$, a north corner always contributes the value of $1$, an east corner always contributes a constant factor of $t^\frac{1}{2}$ and a west corner always contributes a constant factor of $t^{-\frac{1}{2}}$. Consequently, the term $P_{s'}(t)$ is always a monomial. Furthermore, the degree can be readily determined by the difference between the number of east corners and west corners in the assignment of $s'$.

\medskip

It turns out that Theorem \ref{Thm:parallelinvariant} also holds for general spatial MOY graphs that are not necessarily plane. The proof involves a considerably more tricky construction of a correspondence between $S(G, \delta)$ and $S(G', \delta)$, which is typically not bijective, such that the corresponding Kauffman state sums are identical.  Given the lengthy and technical nature of the argument, and the fact that the result is not directly relevant to the primary focus of our paper, we will omit it here and present it in a future paper, hopefully in a more appropriate context.

\section{Clock moves and clock theorem}\label{section:clock}

In the remainder of this paper, $G$ always refers to a \textit{plane} balanced transverse graph (see Definition \ref{Def:MOY}; in particular, $c(e)=1$ for all edges) unless otherwise specified.  We should place the base point $\delta$ on an edge $e_0$ that is adjacent to the \textit{unbounded} region.  Let $r=v_0$ denote the head of $e_0$.  The remaining vertices of $G$ are ordered arbitrarily and denoted by $v_1, v_2, \cdots v_k$. For each vertex $v_i$, $1\leq i \leq k$, assume there are a total number of $d_i$ incoming edges. Basically, $d_i$ is the \textit{degree} of the vertex $v_i$, and the balanced coloring condition implies that there are also $d_i$ outgoing edges from the vertex $v_i$.

Next, we will order the $d_i$ edges entering $v_i$ in counter-clockwise order.  So the ``leftmost'' edge is the first edge; the next is the second edge and so on.  For each integer $x_i$ with $1\leq x_i \leq d_i$, we can talk about the $x_i$-th edge entering the vertex $v_i$, counting counter-clock-wisely, and denote it by $e_{i, x_i}$. We can then associate a subgraph of $G$ with $k$ edges by a lattice point $\vec{x}=(x_1, x_2, \cdots , x_k)$.

\begin{definition}\label{Def:lattice}
Suppose $1\leq x_i \leq d_i$ for $i=1, 2, \cdots, k$. We use $\vec{x}=(x_1, x_2, \cdots , x_k)$  to denote the subgraph $H$ with the edge set $\bigcup\limits_{i=1}^k e_{i, x_i}$, the union of the $x_i$-th edge entering the vertex $v_i$.  We also say that the lattice point $\vec{x}$ \textit{generates} or \textit{represents} $H$.

\end{definition}

Hence, there are a total of $d_1d_2\cdots d_k$ lattice points in the $k$-dimensional space, which gives rise to $d_1d_2\cdots d_k$ distinct subgraphs $H$ of $G$.  

\begin{definition}
Given a vertex $r$ in an oriented graph $G$, an {\it oriented spanning tree of $G$ rooted at $r$} is a spanning subgraph $T$ that satisfies the following three conditions:
\begin{enumerate}
\item[(i)] Every vertex $v\neq r$ has in-degree $1$.
\item[(ii)] The root $r$ has in-degree $0$.
\item[(iii)] $T$ has no oriented cycle.
\end{enumerate}

\end{definition}

As the in-degree of each vertex (except $r$) is 1 for $H$, it follows that:

\begin{proposition}\label{proposition:equivalenceofspanning}
A subgraph $H$ represented by a lattice point $\vec{x}$ is an oriented spanning tree rooted at $r$ 
if and only if the following equivalent conditions hold:

\begin{enumerate}
    \item[(a)] The graph $H$ is connected.

    \item[(b)] The graph $H$ does not have a cycle.
    
\end{enumerate}

\end{proposition}

Let $\mathcal{T}_r(G)$ denote the set of all oriented spanning trees of $G$ rooted at $r$. From this point onwards, any mention of a spanning tree shall be understood to refer to an oriented spanning tree rooted at a fixed vertex $r$.

\begin{definition}
\begin{enumerate}
    \item We denote $\vec{x}\preceq \vec{x'}$ if $x_i\leq x_i'$ for all $i$.

    \item $\vec{x}$ and $\vec{x'}$ are called \textit{neighboring} if there is a unique index $i$ such that $x_i=x'_i\pm 1$; and $x_j=x'_j$ for all other $j$.

    \item The \textit{distance} $d(\vec{x}, \vec{x'})$ between two lattice points is defined as the number of index $i$ such that $x_i\neq x_i'$.  

    \item The \textit{norm} of $\vec{x}$ is  $|\vec{x}|=x_1+x_2+\cdots x_k.$
    
\end{enumerate}
   
\end{definition}

Clearly, $\vec{x}$ and $\vec{x'}$ are neighboring lattice points if and only if $d(\vec{x}, \vec{x'})=1$ and $|\vec{x}|=|\vec{x'}|\pm 1$.  Also, suppose $H$ and $H'$ are the two subgraphs generated by $\vec{x}$ and $\vec{x'}$, respectively.  Then $H$ and $H'$ differ by $m$ edges (or equivalently, they have $k-m$ common edges) if and only if $d(x,x')=m$.

Finally, we shall define the \textit{clock moves} between two spanning trees.  This notion is inspired by Kauffman's work on knot diagrams \cite{MR712133} and serves as a fundamental element in this research.   

\begin{definition}[Clock Moves]

Two spanning trees $T$ and $T'$ are said to be related by a clock move if they are represented by neighbouring lattice points $\vec{x}$ and $\vec{x'}$ respectively.  

\end{definition}

Graphically, a clock move can be viewed as ``rotating'' an edge of one spanning tree to the neighboring edge in the clockwise order. See Figure \ref{fig:clock} for an illustration.

\begin{figure}[!h]
    \centering
    \includegraphics[width=13cm]{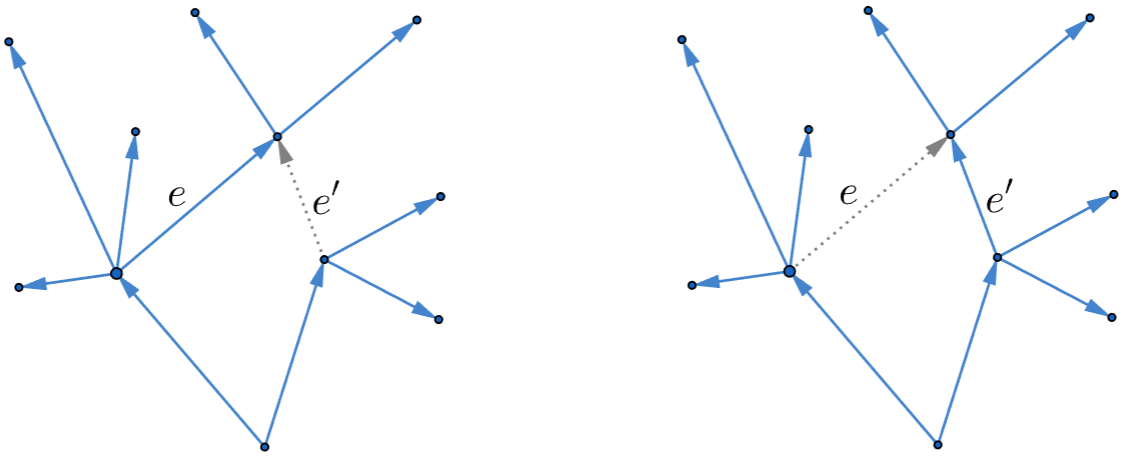}
    \caption{The counter-clockwise clock move turning $e$ to $e'$. Blue edges are contained in the spanning trees while dotted edges are in the complement of the spanning trees.}
    \label{fig:clock}
\end{figure}

\begin{remark}
Although the above edge-rotation move can in principle be defined for arbitrary subgraphs, it is important to note that within this paper, the term ``clock move'' will refer specifically to a move between two spanning trees. 

\end{remark}

As with the Clock Theorem for knot diagrams \cite{MR712133}, we will establish the following analogous result in this section.  

\begin{theorem}[Clock Theorem] \label{Thm:Clock}

Any two spanning trees can be related by a sequence of clock moves.  
    
\end{theorem}

\begin{proposition}\label{Prop:partial}
Suppose $\vec{x}\preceq \vec{x'}$ and both generate spanning trees.  Then, for any intermediate $\vec{x}\preceq \vec{y}\preceq  \vec{x'}$, the subgraph generated by $\vec{y}$ is also a spanning tree.  
    
\end{proposition}

\begin{proof}

Suppose the subgraph $H$ induced by $\vec{y}$ is not a spanning tree. In this case, it must contain a cycle, which we will denote by $C\subset H$. If the induced orientation of $C$ is going counter-clockwise, we claim that for all subgraph $H'$ obtained by a single clockwise clock move from $H$, it must also contain a cycle. Therefore, the subgraph given by $\vec{x}\, (\preceq \vec{y})$ must contain a cycle by induction and contradicts with the assumption that $\vec{x}$ generates a spanning tree.

To prove the above claim, note that in the event that such a clock move does not occur at any vertex of $C$, the cycle $C$ remains a cycle subsequent to the clock move. We now suppose that the clockwise clock move turns an edge $e\in C$ to $e'$. Then the head of $e'$ (denoted by $v$) is bounded inside or on $C$. Recall at the beginning of this section, we fix the root $r$ to be adjacent to the unbounded region.  Thus, the root $r$ lies outside the cycle $C$. We now walk ``backwards'' along $H'$ starting from $v$.  Remember that all vertices except $r$ have in-degree 1 by the construction of $H'$, so the standard argument in graph theory implies that we will eventually revisit some vertices.  This gives the desired cycle in $H'$. (See Figure \ref{fig:counterclosckwisecycle}.)

\begin{figure}[!h] 
    \centering
    \includegraphics[width=6cm]{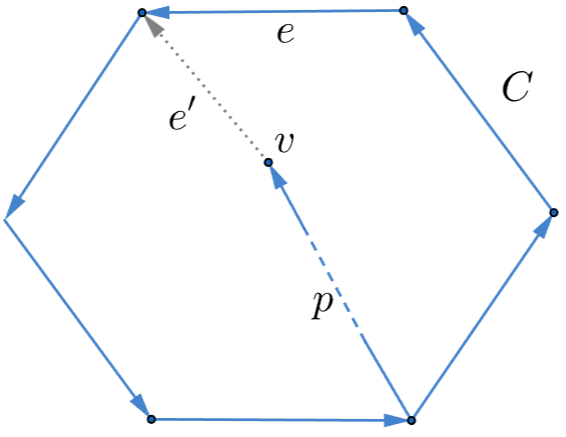}
    \caption{Stating from $v$ and walking backwards, we obtain a new cycle in $H'$ consisting of a path $p$, the ``left'' half of the original cycle $C$, and the new edge $e'$.}
    \label{fig:counterclosckwisecycle}
\end{figure}

Similarly, if the induced orientation of $C$ is going clockwise, then the subgraph given by $\vec{x'} \,(\succeq \vec{y})$ must contain a cycle.  In either case, we arrive at a contradiction.

\end{proof}

\begin{lemma} \label{lemma:2edgediff}
Let $T$ and $T'$ be two spanning trees that differ by two edges, that is, they have $k-2$ common edges.  Then there exists a spanning tree, denoted by $H$, which differs from both $T$ and $T'$ by a single edge.   
  
\end{lemma}

\begin{proof}

Without loss of generality, assume $\vec{x}=(x_1, x_2, \cdots, x_k)$ represents $T$ and $\vec{x'}=(x_1', x_2', \cdots, x_k')$ represents $T'$, where $x_1\neq x_1'$, $x_2\neq x_2'$
and $x_i=x_i'$ for $i=3, 4, \cdots, k$.  In the event that both $x_1<x_1'$ and $x_2<x_2'$ are true (or false), then $\vec{x}\preceq \vec{x'}$ (or $\vec{x'}\preceq \vec{x}$).  Proposition \ref{Prop:partial} then provides the desired spanning tree $H$, represented by $(x_1', x_2, \cdots, x_k)$. 

It thus suffices to consider the case where $x_1>x_1'$ and $x_2<x_2'$. If the subgraph $K$, defined by $(x_1', x_2, \cdots, x_k)$, is a spanning tree, then it will be the desired one, differing from both $T$ and $T'$ by a single edge. Therefore, we assume that $K$ is not a spanning tree and contains a cycle $C$. Let $e_i$ (resp. $e_i'$) be the $x_i$-th (resp. $x_i'$-th) edge pointing at $v_i$. We have $e_1', e_2 \subset C$ since both $(x_1, x_2, \cdots, x_k)$ and $(x_1', x_2', \cdots, x_k')$ induce spanning trees. Since $e_1'$ is obtained from a clockwise clock move on $e_1$, and $e_1$ lies outside the region bounded by $C$ (otherwise $T$ contains a cycle by the claim in the proof of the previous proposition), we see that $C$ must be counter-clockwise (See Figure \ref{fig:twodifference}). We then proceed to prove that $(x_1, x_2', x_3, \cdots, x_k)$ gives the desired spanning tree.  

\begin{figure}[h!] 
    \centering
    \includegraphics[width=8cm]{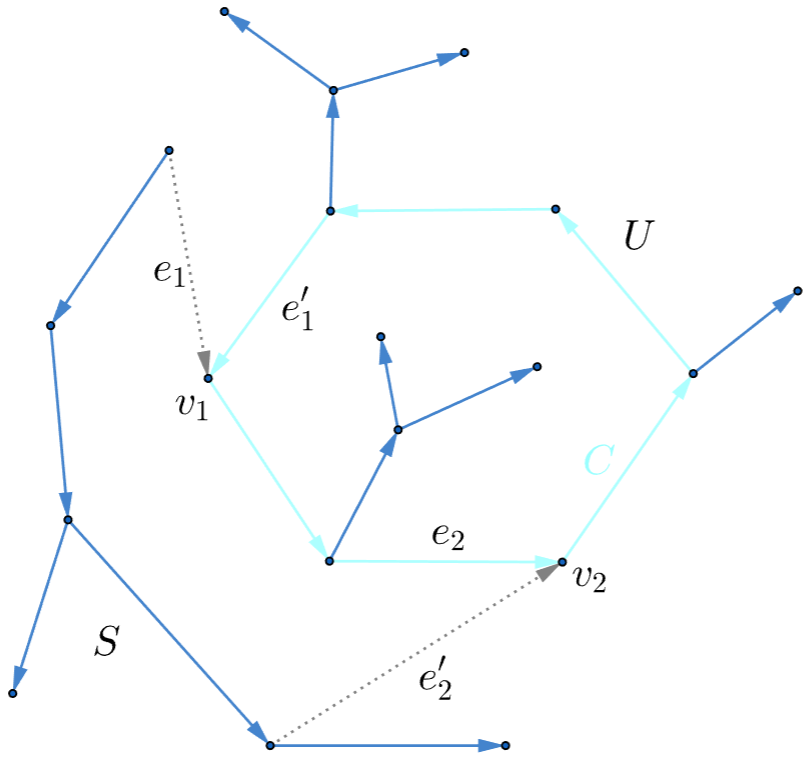}
    \caption{If the graph $K$ represented by $(x_1', x_2, \cdots, x_k)$ contains a cycle, then the graph $H=K\cup\{e_1, e_2'\}\backslash \{e_1', e_2\}$ represented by $(x_1, x_2', \cdots, x_k)$ is a spanning tree. Light blue emphasizes the cycle C. }
    \label{fig:twodifference}
\end{figure}

By Proposition \ref{proposition:equivalenceofspanning}, $K$ is disconnected. Moreover, $K$ has two connected components since replacing $e_1'\subset K$ by $e_1$ (which results in $T$) can at most reduce the total number of components by $1$. Let $U$ be the component containing $C$ and let $S$ be the other component. The fact that $T$ is a spanning tree implies that $U$ and $S$ can be connected by $e_1$; the fact that $T'$ is a spanning tree implies that $U$ and $S$ can be connected by $e_2'$. Since $e_1',e_2\subset C\subset U$, we have $U\backslash e_1'$ is still connected and $U\backslash (e_1'\cup e_2)$ has two components. The component containing $v_1$ can be connected with $S$ by joining $e_1$; the other component containing $v_2$ can be connected with $S$ by joining $e_2'$. Thus by replacing $e_1'$ by $e_1$ and $e_2$ by $e_2'$ in $K$, we obtain a connected subgraph, which is a spanning tree by Proposition \ref{proposition:equivalenceofspanning}.
    
\end{proof}

In fact, the idea of the above proof can be applied to a more general situation. 

\begin{lemma} \label{lemma:m-edgediff}
Let $T$ and $T'$ be two spanning trees that differ by $m\geq 2$ edges.  Then there exists a spanning tree $H$ that differs from both $T$ and $T'$ by at most $m-1$ edge.  

\end{lemma}

\begin{proof}

Without loss of generality, assume $x_i\neq x_i'$ for $i=1,2, \cdots m$. Consider the subgraph generated by $(x_1', x_2, \cdots , x_k)$.  If it is a spanning tree, then that is the desired spanning tree of the lemma.  So let us assume that the subgraph contains a cycle $C$. In accordance with the earlier notations, the components $U$ and $S$, which contain and do not contain $C$, respectively, within the subgraph, will be denoted as such. Since $(x_1',x_2',\cdots, x_m',x_{m+1},\cdots x_k)$ represents a spanning tree, there exists some $e_i\subset C$ for some $2\leq i \leq m$ such that $e_i'$ connects $U$ and $S$. We may assume $i=2$. Then we apply the same arguments from the previous lemma here. Consequently, the subgraph $H$ generated by $(x_1,x_2',x_3,\cdots,x_k)$ is connected, and therefore a spanning tree. It differs from $T$ by one edge and differs from $T'$ by $m-1$ edges.

\end{proof}

\begin{proof}[Proof of Theorem \ref{Thm:Clock}]

Suppose $T$ and $T'$ are two spanning trees that differ by $m$ edges.  We prove by induction on $m$.  For $m=1$, Proposition \ref{Prop:partial} implies that all subgraphs generated by the intermediate lattice points $\vec{y}$ are also spanning trees. Moreover, by definition, the spanning trees generated by neighboring points are related by a clock move.  Now, suppose the statement is proved for all integers less than $m$. By Lemma \ref{lemma:m-edgediff}, there is a spanning tree, denoted $H$, that differs from both $T$ and $T'$ by at most $m-1$ edges.  By the induction hypothesis, a sequence of clock moves can be constructed that relates $H$ and $T$, as well as $H$ and $T'$.  The desired sequence relating $T$ and $T'$ can then be obtained by combining the two. This concludes the induction argument.

\end{proof}

\section{A spanning tree model}

Given a plane graph $G$, recall that each Kauffman state $s$ contributes a monomial $P_s(t)$ to the Alexander polynomial: $$\Delta_G(t)=\sum_{s\in \mathcal{S}(G,\delta)}P_s(t).$$ As there is a bijective map \begin{equation} \label{phi}
   \phi: \mathcal{T}_r(G) \rightarrow S(G, \delta) 
\end{equation}
associating a spanning tree $T$ to the Kauffman state $s$, we readily see that $\Delta_G(1)$ counts the number of spanning trees.  In this section, we will also determine the degree of the monomial $P_s(t)$ from the spanning tree.

Let us first give an explicit description of the bijective map $\phi$. Recall that each plane graph $G$ has a dual graph $G^*$ that has a vertex for each face of $G$ and has an edge $e^*$ connecting the two faces adjacent to $e$.
For every spanning tree $T$ of $G$, it is well known that the complement of the edges dual to $T$ is a spanning tree for $G^*$. This spanning tree is denoted $T^*$ and is referred to as the \textit{dual spanning tree}. 
Note that if the vertex dual to unbounded region in $G$ is fixed to be the root $r^*$ of $G^*$, then there is a unique induced orientation on $T^*$ so that all vertices except the root $r^*$ has in-degree 1. The Kauffman state $s$ is then constructed as follows:

\begin{enumerate}
\item For the edge $e_0$ where the base point  $\delta$ is located, assign the crossing $cr(e_0)$ to its north corner inside the circle region around the vertex $r$.
\item For each edge $e$ in the oriented spanning tree $T$ that enters the vertex $v$, assign the crossing $cr(e)$ to its north corner inside the circle region around the vertex $v$.
\item For all other crossings, there is a unique way of assigning one of the east and west corners: Starting from the root $r*$, one can travel to all other vertices in $G^*$ along edges of $T^*$.  In each step that we traverse $e$ on the dual edge $e^*$ from $v_1^*$ and $v_2^*$, assign $cr(e)$ to the corner that belongs to the regular region dual to $v_2^*$. See Figure \ref{fig:phimap}.


\end{enumerate}


\begin{figure}[h!]
    \centering
    \includegraphics[width=13cm]{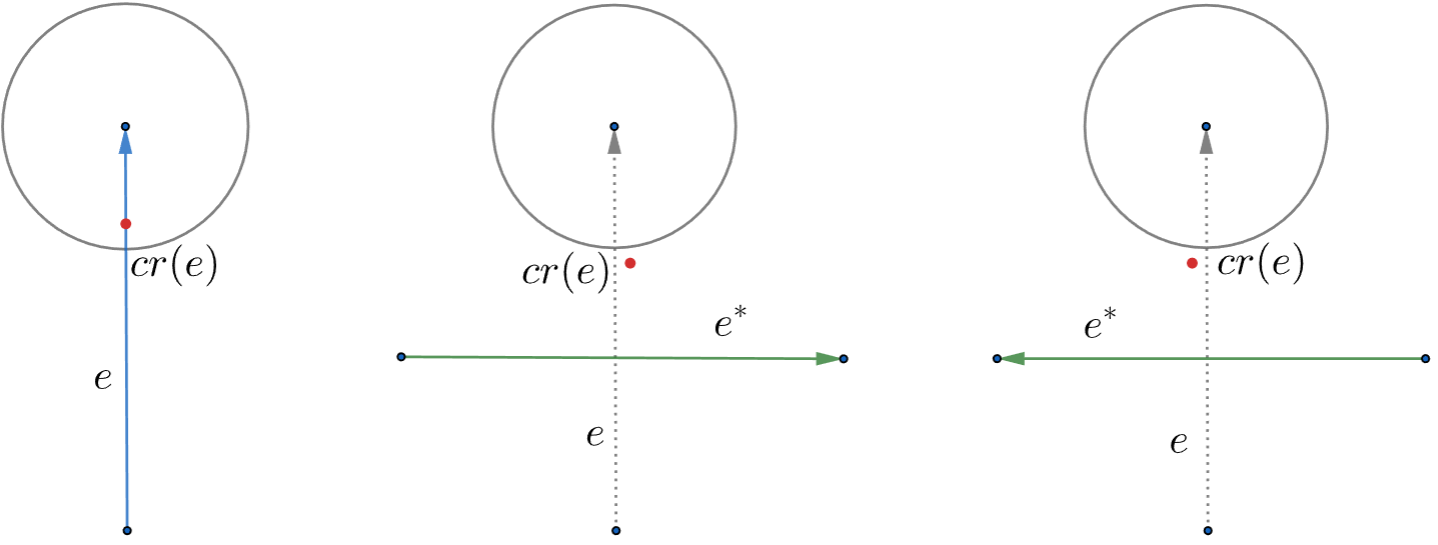}
    \includegraphics[width=14cm]{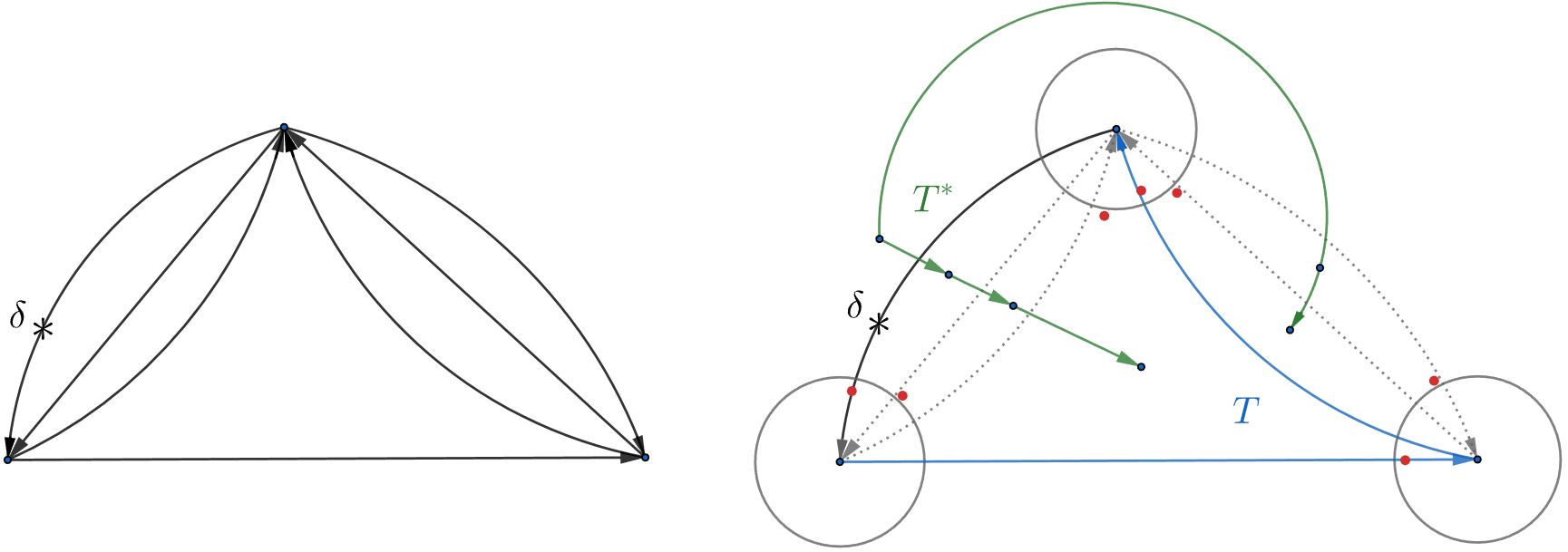}
    \caption{As a convention of this paper, blue edges denote the spanning tree, dotted edges the complement of the spanning tree, and green edges the dual tree. A red dot at a corner of each crossing represents a Kauffman state. Top left: $e$ is an edge of the spanning tree if and only if $cr(e)$ is assigned to its north corner. Top middle: $e^*$ passes $e$ from left to right if and only if $cr(e)$ is assigned to its east corner. Top right: $e^*$ passes $e$ from right to left if and only if $cr(e)$ is assigned to its east corner.
    Bottom: A concrete example of the correspondence between a spanning tree and a Kauffman state.}
    \label{fig:phimap}
    \end{figure}

Our main theorem is the following spanning tree model for the Alexander polynomial.

\begin{theorem} [Spanning Tree Model]
\label{Thm:spanningtreemodel}
Suppose $G$ is a balanced transverse plane graph, and $X$ is the set of lattice points $\vec{x}$ that represent rooted spanning trees.  Then, we have 

\begin{equation}\label{eq:spanning}
\Delta_G(t)\doteq\sum _{\vec{x}\in X}  t^{|\vec{x}|}
\end{equation}

\end{theorem}

Since two neighboring points $\vec{x}$ and $\vec{x'}$ satisfy $|\vec{x}|=|\vec{x'}|\pm 1$, and the Clock Theorem \ref{Thm:Clock} implies that all spanning trees are related by clock moves, Theorem \ref{Thm:strictpositivity} follows as an immediate consequence.  

The proof of Theorem \ref{Thm:spanningtreemodel} is based on the following lemma. 

\begin{lemma} \label{Lemma:degreedifference}

Suppose $\vec{x}$ and $\vec{x'}$ are two lattice points that represent the spanning trees $T$ and $T'$ in $G$, and they further correspond to the Kauffman states $s$ and $s'$ under the canonical bijection $\phi: \mathcal{T}_r(G) \rightarrow S(G, \delta)$.  Then, the difference in the degree of the induced monomials in $t$ is given by:
$$\deg P_s(t)-\deg P_{s'}(t)=|\vec{x}|-|\vec{x'}|$$
\end{lemma}

By the Clock Theorem \ref{Thm:Clock}, it suffices to prove the lemma for two neighboring lattice points $\vec{x}$ and $\vec{x'}$.  Assume without loss of generality that $x_1=x_1'+1$ and $x_i=x_i'$ for $i=2, \cdots k$, so the two spanning trees $T$ and $T'$ differ by a single edge entering the vertex $v_1$.  Denote these two edges $e_1$ and $e_1'$, and they are the $x_1$-th and $x_1'$-th edges towards $v_1$, respectively.  

Note that $T\cup \{e_1'\}=T'\cup \{e_1\}$ is a subgraph with $k+1$ edges, and hence it contains a unique (unoriented) cycle $C$ which must include $e_1$ and $e_1'$ as consecutive edges.  Topologically, $C$ is a simple closed curve, so it divides the plane into an interior bounded region and an exterior unbounded region.  We then classify all clock moves into one of the following two types.  

\begin{definition} A clock move is called

\begin{enumerate}
    \item \textit{Local/Interior}: if the interior region bounded by $C$ does not contain any other edges incident to $v_1$ other than $e_1$ and $e_1'$. See Figure \ref{fig:localmove}.\\

    \item \textit{Global/Exterior}: if the exterior unbounded region does not contain any other edges incident to $v_1$ other than $e_1$ and $e_1'$. See Figure \ref{fig:globalm}.
    
\end{enumerate}
\end{definition}
\begin{figure}[!h]
    \centering
    \includegraphics[width=9cm]{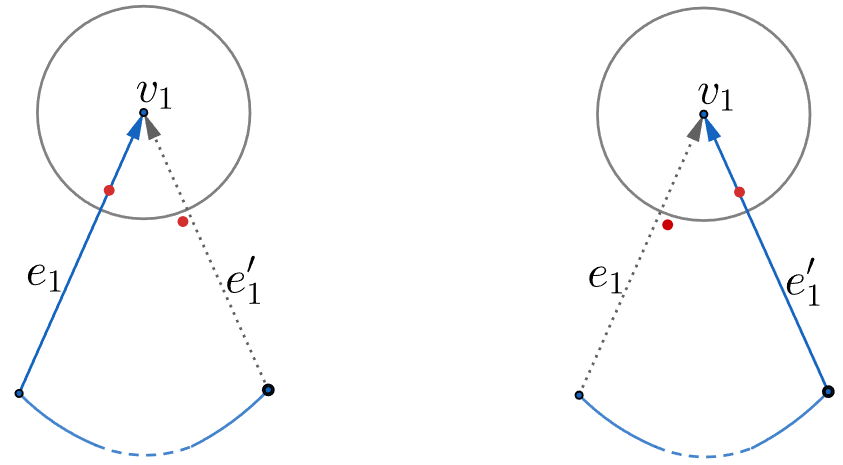}
    \caption{A local move turning $e_1$ to $e_1'$.  The corresponding Kauffman states $s$ and $s'$ are represented by red dots and are identical everywhere except at the two crossings generated by $e_1$ and $e_1'$.}
    \label{fig:localmove}
\end{figure}

\begin{figure}[!h]
    \centering
    \includegraphics[width=10cm]{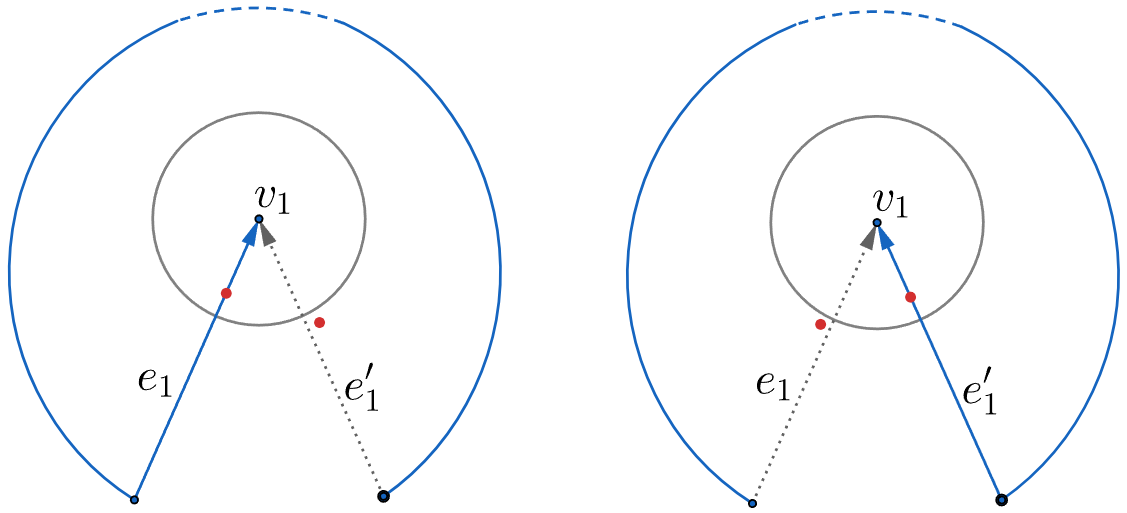}
    \caption{A global move turning $e_1$ to $e_1'$. The corresponding Kauffman states $s$ and $s'$ differ not only at the two crossings generated by $e_1$ and $e_1'$ but also elsewhere (not shown in the picture).}
    \label{fig:globalm}
\end{figure}

Although the two types of clock moves have seemingly symmetric definitions, a local move is, in fact, much simpler in many respects than a global move.  The contrast between the two types will be made apparent when we examine the corresponding clock moves on Kauffman states under the bijective map $\phi: \mathcal{T}_r(G) \rightarrow S(G, \delta)$ described earlier.

First, the local clock moves on Kauffman states can be defined in a straightforward manner. 

\begin{definition}
Two Kauffman states $s$ and $s'$ are related by a local clock move if they differ only at the two crossings generated by the neighboring edges $e_1$ and $e_1'$ as depicted in Figure \ref{fig:localmove}.

\end{definition}

In contrast, a global clock move involves a change of assignment at crossings on a large scale. Note that the dual spanning tree $T^*$ and $T'^*$ also differs by a single edge, so their union also contains a unique cycle, denoted $C^*$.  Let $\mathcal{E}$ be the set of edges in the original graph $G$ that intersects $C^*$.

\begin{definition}
Two Kauffman states $s$ and $s'$ are related by a global clock move if they have different corner assignments at all crossings generated by edges in $\mathcal{E}$ as depicted in Figure \ref{fig:divergence}. We also refer to Figure \ref{fig:globalmove} for a concrete example.
    
\end{definition}

\begin{figure}[!h]
    \centering
    \includegraphics[width=14cm]{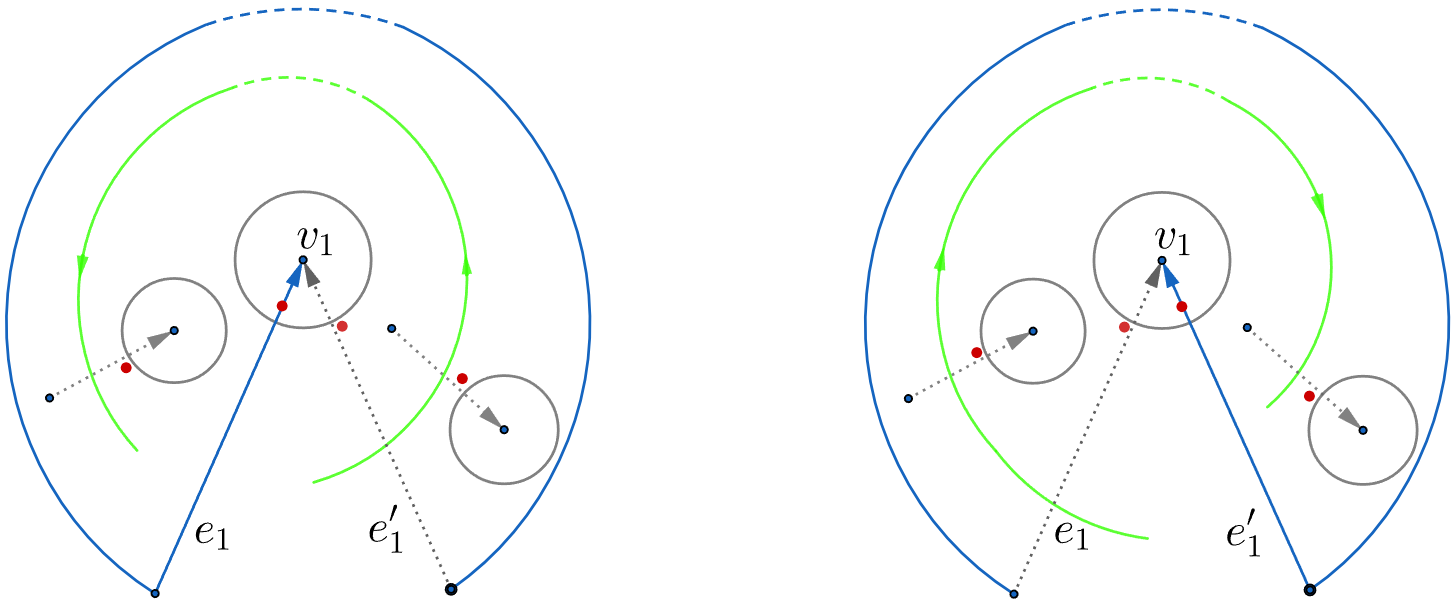}
    \caption{The global move turning $e_1$ to $e_1'$ is counter-clockwise. The light green curve on the left is $C^*\backslash e_1^*$ and that on the right is $C^*\backslash e_1'^*$. If an edge points in (resp. out of) the region bounded by $C^*$, then the generated crossing is assigned to its east (resp. west) corner. 
    }
    \label{fig:divergence}
\end{figure}

\begin{proof}[Proof of Lemma \ref{Lemma:degreedifference}]\label{proof:degreedifference}

As explained, it suffices to prove the lemma for two neighboring lattice points $\vec{x}$ and $\vec{x'}$ with $x_1=x_1'+1$. The clock move is either local or global. 

In the case of a local move, the Kauffman state $s$ has one more east corner at $e_1'$ relative to $s'$, whereas $s'$ has one more west corner at $e_1$ relative to $s$. As the east corner contributes a factor $t^{\frac{1}{2}}$ while a west corner contributes $t^{-1/2}$, we readily see that $\deg P_s(t)-\deg P_{s'}(t)=1=|\vec{x}|-|\vec{x'}|  $.

In the case of a global move, the idea of the proof is similar: it will be demonstrated that the overall number of east corners (resp. west corners) of the Kauffman state $s$ is one more (resp. one less) than that of the Kauffman state $s'$. This is evidently equivalent to the claim that among the set of edges $\mathcal{E}$ intersecting $C^*$, exactly half of the edges leave the region bounded by $C^*$, while the remaining half of the edges enter that region. See Figure \ref{fig:divergence}.

In essence, the above statement can be regarded as a discrete version of the \textit{divergence theorem} in calculus: given that the in-degree and out-degree of each vertex are equal for our graph $G$, it is reasonable to expect a balance in the total inflow and outflow with respect to a closed curve $C^*$. More concretely, note that the total in-degree and out-degree of all vertices inside $C^*$ are equal.  When the total number of edges inside $C^*$ is subtracted from the total in-degree and out-degree, the resulting value represents the number of edges entering or leaving the interior region, respectively, and they are equal.     

\end{proof}

\begin{proof}[Proof of Theorem \ref{Thm:spanningtreemodel}]
Recall that $\Delta_G(t)=\sum_{s\in \mathcal{S}(G,\delta)}P_s(t)$.  Under the bijective map $\phi^{-1}: S(G, \delta)\rightarrow \mathcal{T}_r(G) $ and the representation of spanning trees as lattice points (Definition \ref{Def:lattice}), we have 
$$\sum_{s\in \mathcal{S}(G,\delta)}P_s(t)\doteq\sum _{\vec{x}\in X}  t^{|\vec{x}|}$$ since each term $P_s(t)$ has relative degree $|\vec{x}|$ by Lemma \ref{Lemma:degreedifference}.

\end{proof}


\medskip

In general, we remark that the spanning tree model can also be used to study the Alexander polynomial of a plane MOY graph $(G,c)$ with a potentially non-trivial coloring $c$. To this end, we can simply consider the graph $G'$ obtained from parallel edge replacements on all edges of $G$, and Theorem \ref{Thm:parallelinvariant} implies that $\Delta_G(t)=\Delta_{G'}(t)$.

\begin{corollary}
    For all $n\in\mathbb{Z}$ with $n\geq 2$, $n$ cannot be realized as the Alexander polynomial of any plane MOY graphs.
\end{corollary}

\begin{proof}
    Suppose there is such a graph $G$ with $\Delta_G(t)\doteq n$. By Theorem \ref{Thm:parallelinvariant}, we may assume that $G$ is equipped with a trivial coloring. By \cite[Theorem 1.3]{BWSpanning}, $G$ has $\Delta_G(1)=n$ spanning trees. Furthermore by Theorem \ref{Thm:spanningtreemodel}, all the spanning trees have the same degree/norm. On the other hand, Clock Theorem (Theorem \ref{Thm:Clock}) states that any two spanning trees can be connected by a sequence of clock moves. In particular, since $n\geq 2$, there must exist a clock move between some two spanning trees. Lemma \ref{Lemma:degreedifference} implies that such a clock move changes the degree by $1$, so they should contribute two terms of different degrees in the Alexander polynomial, a contradiction.
\end{proof}

\bigskip

\begin{example}\label{exp:lattice}
In Figure \ref{fig:example}, $G$ is an MOY graph with non-trivial coloring. By theorem \ref{Thm:parallelinvariant}, we may replace all edges of $G$ with color $n$ by a set of $n$ parallel edges and obtain a balanced transverse graph $G'$ with $\Delta_G(t)=\Delta_{G'}(t)$. 

\begin{figure}[!h]
    \centering
    \includegraphics[width=5cm]{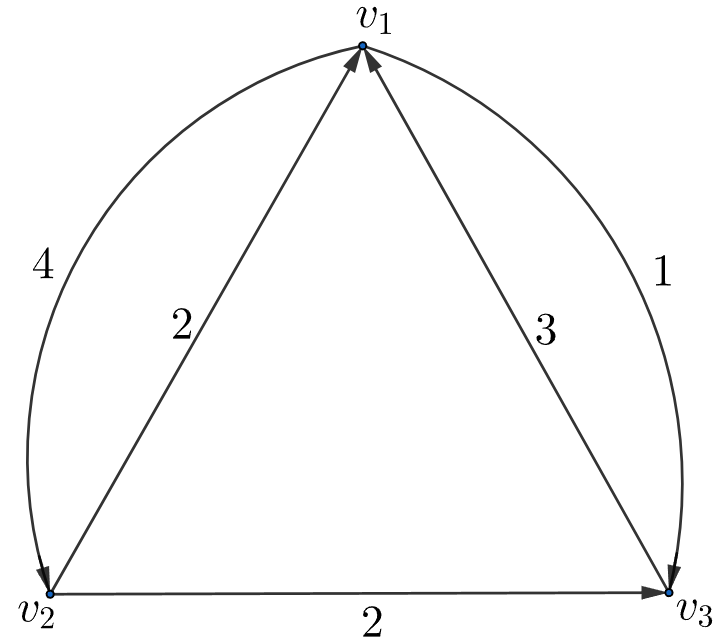}
    \caption{An MOY graph $G$ with non-trivial coloring. The corresponding graph $G'$ is obtained by parallel edge replacements (not shown in the picture). }
\label{fig:example}
\end{figure}

\begin{enumerate}
\item Suppose the base point $\delta$ is placed on the bottom horizontal edge pointing at $v_3$. Then all the spanning trees are rooted at $v_3$. There is only one state hence only one spanning tree for $G$ as shown in Figure \ref{fig:example1}.

\begin{figure}[!h]
    \centering
    \includegraphics[width=6cm]{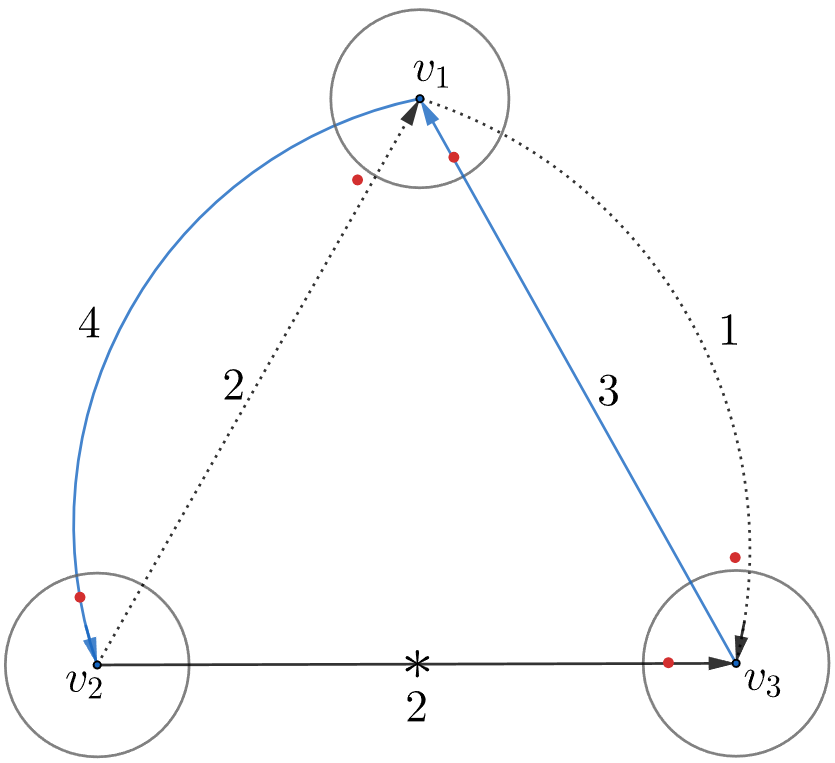}
    \caption{With root $r=v_3$, there is a single spanning tree of weight $3 \times 4= 12$ in $G$. Correspondingly in $G'$, there are $3 \times 4= 12$ spanning trees given by different combinations of parallel edges.   
    }
    \label{fig:example1}
\end{figure}

This corresponds to $3\times 4=12$ spanning trees in $G'$. Specifically in terms of lattice points, for all $1\leq x_1\leq 3$, $1\leq x_2\leq 4$, $(x_1,x_2)$ represents a spanning tree, and these $12$ lattice points have the shape of a $3\times 4$ rectangle (Figure \ref{fig:lattice1}). Note that there always exist local moves between adjacent parallel edges, so the $12$ states (spanning trees) are related by local moves around either $v_1$ or $v_2$.

\begin{figure}[!h]
    \centering
    \includegraphics[width=6cm]{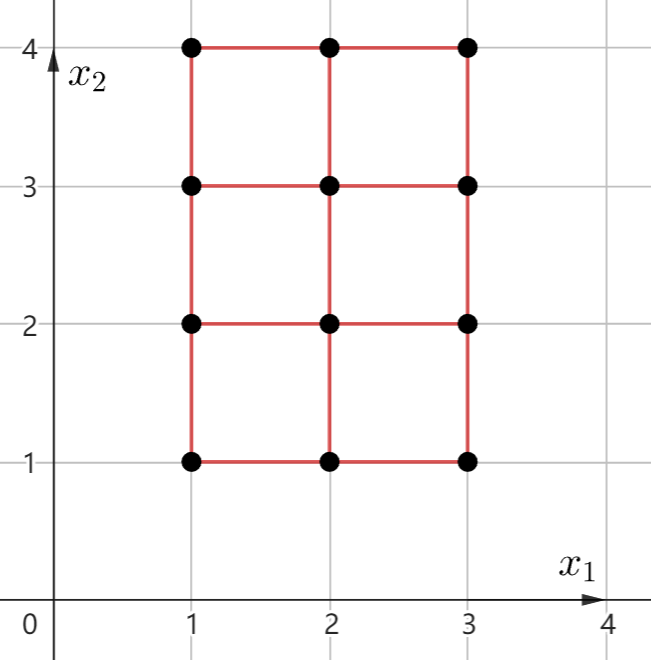}
    \caption{Lattice points representing spanning trees in $G'$ rooted at $r=v_3$. Red edges connect lattice points of spanning trees related by a local move.}
    \label{fig:lattice1}
\end{figure}

\medskip
\item Instead, the base point can be placed on the edge pointing at $v_2$ (Figure \ref{fig:example2}). Then all the spanning trees are rooted at $v_2$. In this case, there are $3$ spanning trees in $G$. 

\begin{figure}[!h]
    \centering
    \includegraphics[width=15cm]{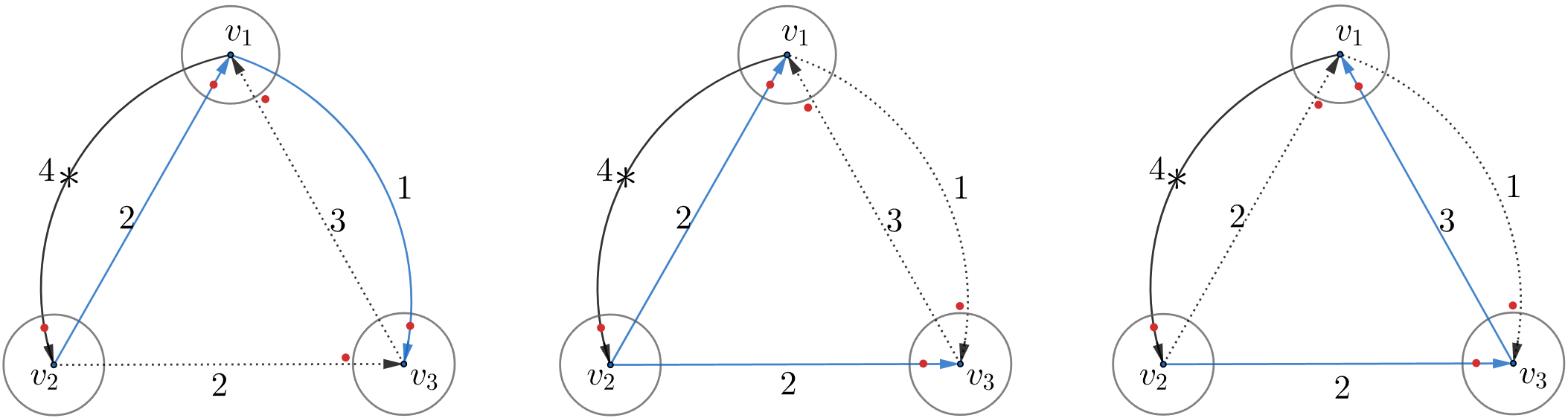}
    \caption{With root $r=v_2$, the total weighted sum of spanning trees in $G$, which equals the total number of spanning tree in $G'$, is $2\times 1+2\times 2 + 2\times 3=12$.}
    \label{fig:example2}
\end{figure}

Here, the first and the second spanning trees are related by a global move, while the second and the third spanning trees are related by a local move. Following the parallel edge replacements, the $3$ spanning trees are respectively associated with $2\times 1=2$, $2\times 2=4$, and $2\times 3=6$ spanning trees in $G'$. If we draw all the lattice points $(x_1, x_3)$ that represent these spanning trees, we see that they are the union of 3 rectangles (Figure \ref{fig:lattice}). Note that all lattice points in any single one of the rectangles are related by local moves. Furthermore, the first and the second rectangles are related by global moves, while the second and third rectangles are related by local moves. 

\begin{figure}[!h]
    \centering
    \includegraphics[width=10cm]{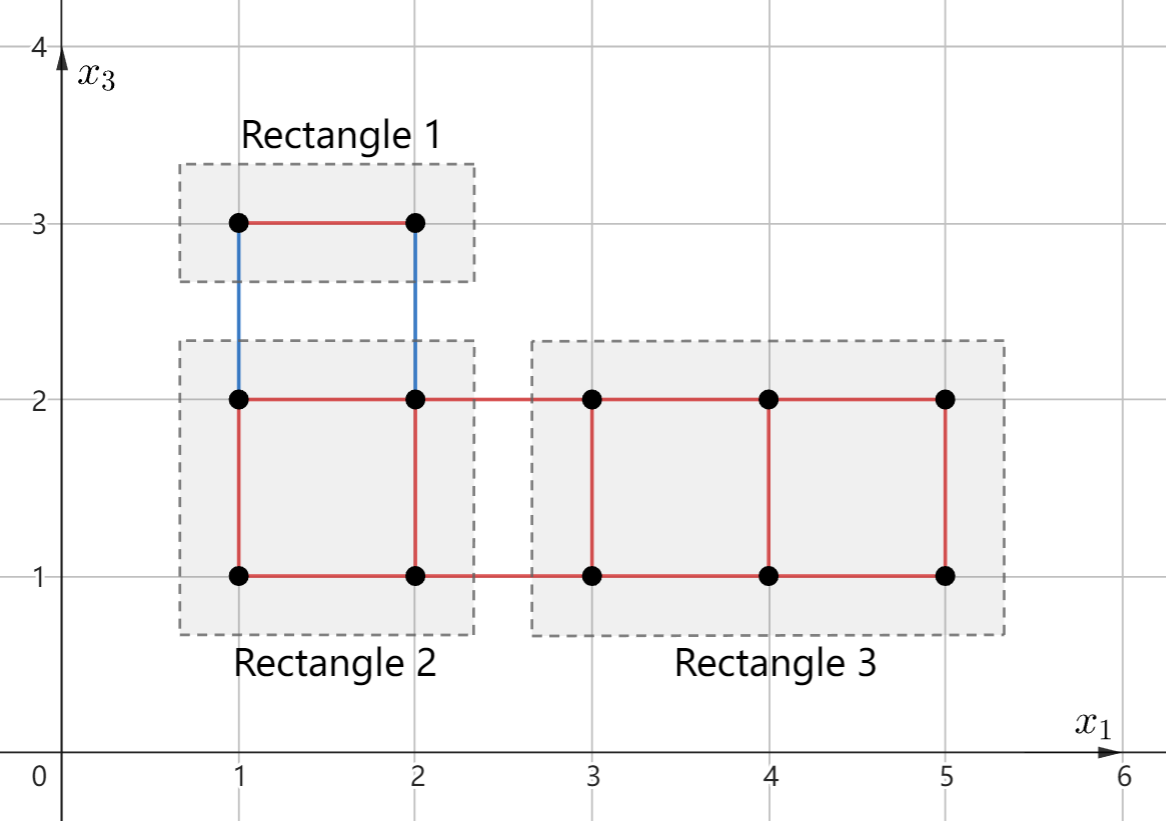}
    \caption{Lattice points representing spanning trees in $G'$ rooted at $r=v_2$. Red edges connect lattice points of spanning trees related by a local move, while blue edges connect those related by a global move.}
    \label{fig:lattice}
\end{figure}

\end{enumerate}

\medskip
From the lattice points for both examples, we see that although the shape of the lattice points representing spanning trees can be very different for different choices of roots, the Alexander polynomial $\Delta_G(t)\doteq t^2+2t^3+3t^4+3t^5+2t^6+t^7$ remains invariant.  

\end{example}

\section{Trapezoidal conjecture for plane graphs}

Given $\Delta_G(t)=\sum _{\vec{x}\in X}  t^{|\vec{x}|}$ by Theorem \ref{Thm:spanningtreemodel}, the trapezoidal conjecture can be equivalently viewed as a desired property on the number of lattice points in $X$ that represent spanning trees with a fixed norm $|\vec{x}|$. 
Our strategy is to decompose $X$ as a disjoint union  $X=\sqcup_\alpha X_\alpha$ so that
\begin{enumerate}
    \item[(i)] Each $X_\alpha$ is a \textit{rectangular} sublattice;
    \item[(ii)] The \textit{average norm} of all points in each  rectangle $X_\alpha$ is independent of $\alpha$.
\end{enumerate}
Since the lattice points in a rectangle satisfy the trapezoidal property, while all rectangles share the same ``symmetric axis'' in the trapezoidal configuration by (ii), it follows that the union must also satisfy this unimodal property. 

\begin{definition}

Two spanning trees $T$ and $T'$ are \textit{locally-clock-equivalent} if they are related by a sequence of local clock moves. In this case, the corresponding lattice points $\vec{x}, \vec{x'} \in X$ are also called locally-clock-equivalent.  

\end{definition}

If we denote $T\sim T'$ for two locally-clock-equivalent spanning trees, then $\sim$ is an equivalent relation. Similar things can be said for $\vec{x}\sim \vec{x'}$, which also defines an equivalent relation in $X$.  As a result, we can decompose $X=\sqcup_\alpha X_\alpha$ where each $X_\alpha$ is a locally-clock-equivalent class.  Our first observation is that although $X$, the set of lattice points that represent spanning trees, may have strange shapes, each locally-clock-equivalent class $X_\alpha$ will be a rectangle of dimension at most $k$.

\begin{proposition}\label{proposition:rectangledecomposition}
Each locally-clock-equivalent class $X_\alpha$ is a rectangular sublattice. More precisely, there are integers $1\leq m_{\alpha, i}\leq M_{\alpha, i} \leq d_i$ for $i=1,2\cdots k$, such that 
\begin{equation} \label{eq:equivalentclass}
X_\alpha=\{\vec{x} \, | \, m_{\alpha, i} \leq x_i \leq M_{\alpha, i} \}.
\end{equation}

\end{proposition}

\begin{proof}

Suppose $\vec{x}, \vec{x'}, \vec{x''} \in X_\alpha$, and $\vec{x'}$ (resp. $\vec{x''}$) is obtained from a local clock move at the index $i$ (resp. $j$) on $\vec{x}$. We claim that an additional local clock move at the index $j$ on $\vec{x'}$ and an additional local clock move at the index $i$ on $\vec{x''}$ may be performed, resulting in the same final point $\vec{x'''}\in X_\alpha$. 

This fact is most easily discernible by examining the respective local clock moves on Kauffman states (Figure \ref{fig:localmove}). Since the local clock moves only affects the corner assignment of the Kauffman state locally at the vertex $v_i$ and $v_j$, respectively, we can simply stack these two moves at an arbitrary order to obtain another Kauffman state.     

Note that $\vec{x}, \vec{x'}, \vec{x''}, \vec{x'''}$ have the shape of a rectangle. This leads to the general fact about $X_\alpha$.  

\end{proof}

It thus makes sense to call a locally-clock-equivalent class $X_\alpha$ a \emph{maximal rectangle}. So far, we have shown that the set of lattice points representing the spanning trees $X$ consists of disjoint maximal rectangles $X_\alpha$'s. By (\ref{eq:equivalentclass}), we see that the average norm of all the points $\vec{x} $ in each $X_\alpha$, denoted by $A(X_\alpha)$, is
$$A(X_\alpha)=\frac{1}{|X_\alpha|}\cdot\sum_{\vec{x}\in X_\alpha} |\vec{x}|=\frac {1}{2}\sum_{i=1}^k
(m_{\alpha, i} +M_{\alpha, i}).$$

Observe that given an arbitrary $\vec{x}\in X_\alpha$, the entire $X_\alpha$ can be determined from $\vec{x}$ provided the following information is given: let $T_{\vec{x}}$ be the spanning tree represented by $\vec{x}$. For each $1\leq i\leq k$, let $L_i(\vec{x})\in \mathbb{Z}_{\geq 0}$ (resp. $R_i(\vec{x})$) be the number of local clockwise moves (resp. local counter-clockwise moves) that can be applied on $T_{\vec{x}}$ at the $i$-th vertex. In other words, $L_i(\vec{x})$ (resp. $R_i(\vec{x})$) is equal to the number of points before (resp. after) $\vec{x}$ along the direction of the $i$-th axis in $X_\alpha$. As $m_{\alpha,i}=x_i-L_i(\vec{x})$ and $M_{\alpha,i}=x_i+R_i(\vec{x})$, we obtain another formula that computes the average norm 

    \begin{equation}\label{eq:average}
    \begin{aligned}
        A(X_\alpha)&=\frac {1}{2}\sum_{i=1}^k
(m_{\alpha, i} +M_{\alpha, i})\\
        &=\sum_{i=1}^k (x_i+\frac{1}{2}\left(R_i(\vec{x})-L_i(\vec{x})\right)).
    \end{aligned}
    \end{equation}

\noindent
Our goal is to show that this value is independent of the choice of maximal rectangles:

\begin{proposition}\label{proposition:sameaveragenorm}
    For any maximal rectangles $X_\alpha$ and $X_\beta$, we have $A(X_\alpha)=A(X_\beta)$.
\end{proposition}

\medskip

\begin{example}

In reference to the earlier Example \ref{exp:lattice}, we determined all lattice points representing spanning trees in $G'$ rooted at $r=v_2$ and all clock moves among them. In Figure \ref{fig:symmetricaxis} below, the red edges represent local moves and blue edges represent global moves. Two lattice points are locally-clock-equivalent if and only if there is a path consisting of red edges connecting them. In this example, there are two maximal rectangles. Both of their average norms are $4.5$ thus they share the same symmetric axis $x_1+x_3=4.5$.

\begin{figure}[!h]
    \centering
    \includegraphics[width=12cm]{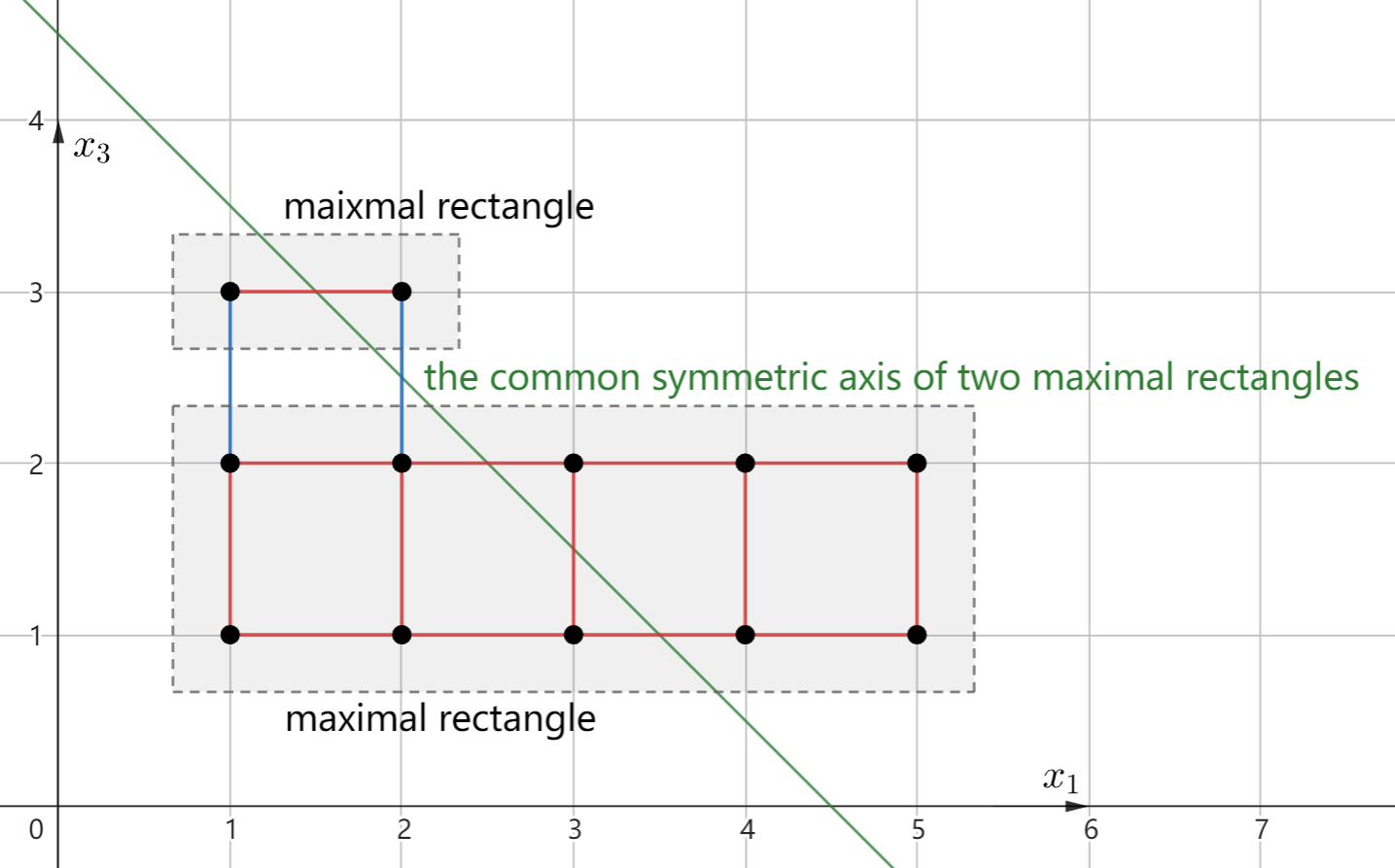}
    \caption{Rectangle 2 and Rectangle 3 in Figure \ref{fig:lattice} merge to form the lower maximal rectangle, which contributes $t^2+2t^3+2t^4+2t^5+2t^6+t^7$, while the upper maximal rectangle contributes $t^4+t^5$ to $\Delta_{G'}(t)$. Each satisfies the trapezoidal property and has an average degree of $4.5$. Their sum, $t^2+2t^3+3t^4+3t^5+2t^6+t^7$, also satisfies the trapezoidal property.}
    \label{fig:symmetricaxis}
\end{figure}

\end{example}

\medskip

Now, suppose $T_{\vec{x}}$ and $T_{\vec{y}}$ are the two spanning trees of $G$ represented by the lattice points $\vec{x}\in X_\alpha$ and $\vec{y}\in X_\beta$, respectively, and there is a counter-clockwise global move turning $T_{\vec{x}}$ to $T_{\vec{y}}$. Without loss of generality, assume this global clock move takes place at $v_1$. Then we have $x_1+1=y_1$ and $x_i=y_i$ for all $i>1$. If we denote by $e_{i,j}$ the $j$-th edge pointing at the $i$-th vertex, this global clock move turns $e_{1,x_1}$ to $e_{1,y_1}$. For simplicity, we will sometimes call an edge $e_{i,j}$ on the left (resp. right) hand side of $e_{i,j'}$ if $j<j'$ (resp. $j>j'$).  Also, we will call a region on the left (resp. right) hand side of an edge $e$ if it is adjacent to $e$ and locates on its clock-wise (resp. counter-clock-wise) side.

Recall that the union of the dual spanning trees $T_{\vec{x}}^*$ and $T_{\vec{y}}^*$ contains a unique cycle, denoted $C^*$. By Jordan curve theorem, $C^*$ divides the plane into an interior part and an exterior part. We say an edge points into (resp. out of) $C^*$ if its head lies in the interior (resp. exterior) part.



\begin{lemma}\label{lemma:consecutive}
    Let $e\in T_{\vec{x}}\cap T_{\vec{y}}$ be a common edge of the two spanning trees.  Then for all the edges on the left-hand side of $e$, either $C^*$ intersects none of them, or the edges intersecting $C^*$ are consecutive. The same result holds for the edges on the right-hand side of $e$.
\end{lemma}

\begin{proof}
    Let $e_1,\cdots, e_n$ be all edges on the left-hand side of $e$ ordered from left to right, and $e_1^*,\cdots,e_n^*$ be their dual edges in $G^*$. Since an edge $e_i$ intersects $C^*$ if and only if $e_i^* \in C^*$, the lemma is equivalent to the claim that such edges $e_i^*$'s are consecutive in $e_1^*\cup \cdots \cup e_n^*$. See Figure \ref{fig:consecutive}.

    \begin{figure}[!h]
    \centering
    \includegraphics[width=15.7cm]{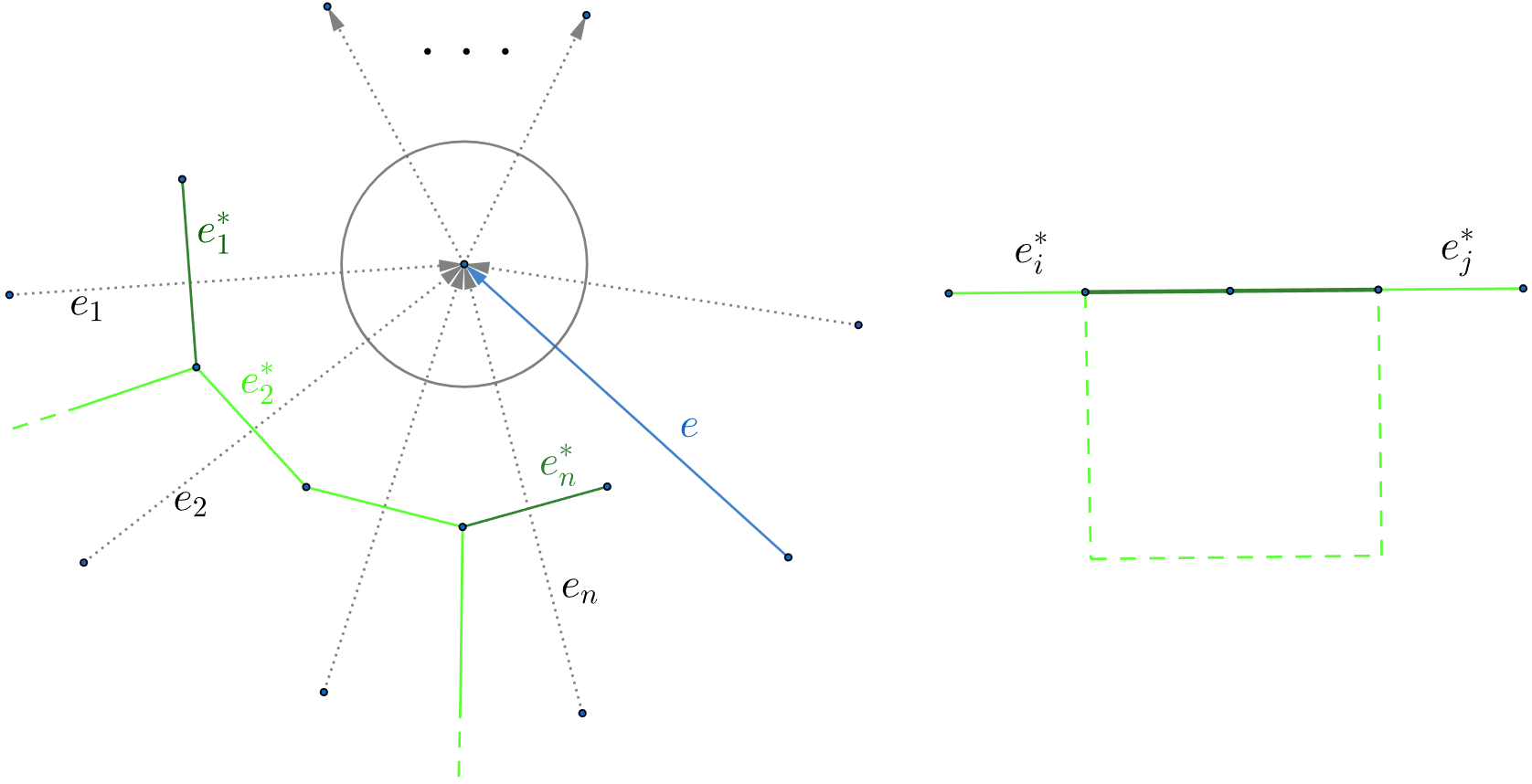}
    \caption{As a convention in this section, we will use light green and dark green to denote edges in $C^*$ and the complement of $C^*$ in $T_{\vec{x}}^*\cup T_{\vec{y}}^*$, respectively.    
    Left: A concrete example. The light green edges are consecutive in $e_1^*\cup \cdots \cup e_n^*$. Right: If the light green edges are not consecutive in $e_1^*\cup \cdots \cup e_n^*$, then there is a cycle formed by some light and dark green edges, a contradiction.}
    \label{fig:consecutive}
\end{figure}

    Suppose they are not consecutive. Then there exists some $1\leq i  < j \leq n$ such that $e_i^*,e_j^*\in C^*$ while $e_k^*\notin C^*$ for all $i<k<j$. Since $C^*$ is connected, there exists a path in $C^*$ connecting $e_i^*$ and $e_j^*$, indicated by the dashed light green path in Figure \ref{fig:consecutive}. By assumption, $e_{i+1}^*\cup \cdots \cup e_{j-1}^*\not\subset C^*$, but it also connects $e_i^*$ and $e_j^*$. Thus, the union of the light green path and $e_{i+1}^*\cup \cdots \cup e_{j-1}^*$ gives another cycle different from $C^*$ in $T_{\vec{x}}^*\cup T_{\vec{y}}^*$, which contradicts the uniqueness of $C^*$.
    
\end{proof}

The same argument applies to the edges pointing at $v_1$, where the global clock move takes place, and we have:

\begin{lemma}\label{lemma:consecutive1}
    Among all edges pointing at $v_1$, the edges intersecting $C^*$ are consecutive.
\end{lemma}


\medskip

Our next goal is to understand the orientations of the edges in $G^*$.  Note that $T_{\vec{x}}^*$ and $T_{\vec{y}}^*$ are rooted trees, so both have an induced orientation on their edges as discussed in the previous section. This, however, does not mean that the induced orientations are always consistent on the common edges of $T_{\vec{x}}^*$ and $T_{\vec{y}}^*$.  Indeed, it is not hard to see that the induced orientations on $T_{\vec{x}}^*$ and $T_{\vec{y}}^*$ are opposite on the common edges in $C^*$ whereas they are the same on the compliment of $C^*$.
To be more explicit, if we let $C_x^*:=C^*\backslash e_{1,x_1}^*$, $C_y^*:=C^*\backslash e_{1,y_1}^*$ and equip them with the induced orientations, then $C^*_x$ goes counter-clockwise relative to $v_1$ while $C^*_y$ goes clockwise. See Figure \ref{fig:globalmove} for an illustration.

\begin{figure}[!h]
    \centering
    \includegraphics[width=11cm]{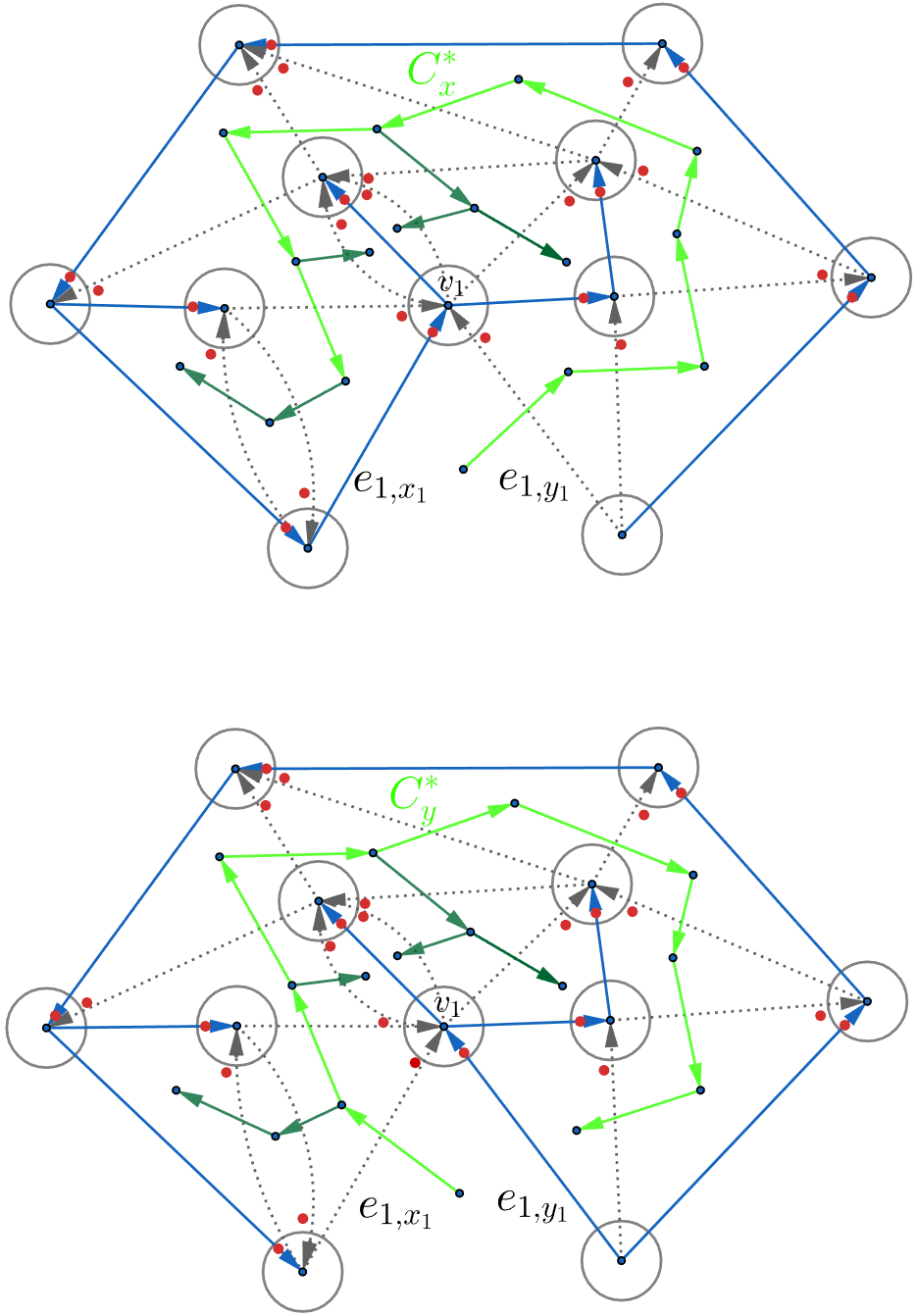}
    \caption{A concrete example of a counter-clockwise global clock move from $\vec{x}$ to $\vec{y}$. Note that the light green $C_x^*$ and $C_y^*$ differ by a single edge and are equipped with opposite induced orientations, while the induced orientations on the dark green edges are the same.}
    \label{fig:globalmove}
\end{figure}

\begin{lemma}\label{lemma:orientations}
    Let $e$ be an edge in $T_{\vec{x}}$ pointing into $C^*$, and $e_1,\cdots,e_n$ be all edges on the left-hand side of $e$ ordered from left to right. 
    Suppose $C^* \cap T_x^*=  e_p^*\cup \cdots\cup e_q^*$ (by Lemma \ref{lemma:consecutive} we know they are consecutive). With the induced orientation on $T_x^*$, the paths $e_p^*\cup \cdots \cup e_q^*$ and $e_{q+1}^*\cup \cdots \cup e_n^*$ (if exists) direct towards $e$ while $e_1^*\cup \cdots \cup e_{p-1}^*$ (if exists) directs away from $e$. See Figure \ref{fig:orientations}.

\end{lemma}

\begin{figure}[!h]
    \centering
    \includegraphics[width=10cm]{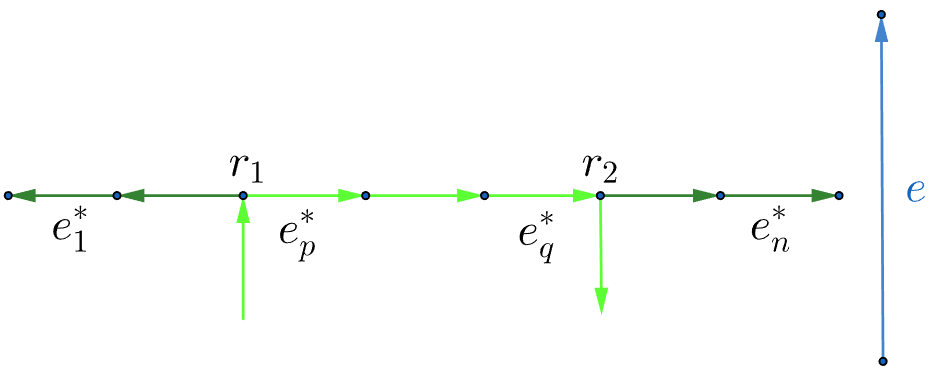}
    \caption{The induced orientation on the dual edges $e_1^*, \cdots, e_n^*$ on the left-hand side of $e$.}
    \label{fig:orientations}
\end{figure}

\begin{proof}
    Since $e$ points into $C^*$ and $C^*_x$ goes counter-clockwise along $C^*$, the path $e_p^*\cup \cdots \cup e_q^*$ must go towards $e$. Using the fact that $T_{\vec{x}}^*$ is a directed tree, we conclude that $e_{q+1}^*\cup \cdots \cup e_n^*$ must also direct towards $e$. 
    
    Let $r_1$ be the tail of $e_p^*$. There exists an edge of $C^*_x$ pointing at $r_1$. Consequently, $e_{p-1}^*\cup \cdots \cup e_1^*$ must direct at the opposite direction since $T_{\vec{x}}^*$ is a directed tree.

\end{proof}

\begin{remark}
Modifying the conditions in Lemma \ref{lemma:orientations} as outlined below results in eight analogous statements:
\begin{enumerate}

\item[(i)] $T_{\vec{x}}$ v.s. $T_{\vec{y}}$; 
\item[(ii)] ``pointing \textit{into}'' v.s. ``pointing \textit{out of}'';
\item[(iii)] ``on the \textit{left} hand side'' v.s. ``on the \textit{right} hand side''.  
    
\end{enumerate}
For simplicity, we will omit the detailed statements and only remark that all the orientations can be explicitly determined by the same methods above.

\end{remark}

\begin{proof}[Proof of Proposition \ref{proposition:sameaveragenorm}]
    By Theorem \ref{Thm:Clock}, it suffices to consider the case where 
    there exists some $\vec{x}\in X_\alpha$ and $\vec{y}\in X_\beta$ such that $\vec{x}$ and $\vec{y}$ are related by a global clock move. Without loss of generality, assume there is a counter-clockwise move applying on the vertex $v_1$ that turns $\vec{x}$ to $\vec{y}$, and hence $\vec{y}=(x_1+1,x_2,x_3,\dots,x_k)$. 
    Recall that we have used earlier the notation $\mathcal{E}$ to denote the set of edges in the graph $G$ that intersects $C^*$. 
    Let $\mathcal{E}_i$ denote the set of edges pointing at $v_i$ that intersect $C^*$. Then $\mathcal{E}=\sqcup_{i=1}^k \mathcal{E}_i$.  Our strategy is to relate the average norm of $X_\alpha$ and $X_\beta$ in terms of  $|\mathcal{E}_i|$ using Formula (\ref{eq:average}).

\bigskip\noindent
 First, we determine the relations between $R_1(\vec{x})$, $L_1(\vec{x})$ and $R_1(\vec{y})$, $L_1(\vec{y})$, the number of local clock moves that can be applied on $T_{\vec{x}}$ and $T_{\vec{y}}$ at the vertex $v_1$.
    By assumption, the clock move turning $e_{1,x_1}$ to $e_{1,y_1}=e_{1,x_1+1}$ is a global counter-clockwise move.  Hence, the clock move is not local, and so $R_1(\vec{x})=L_1(\vec{y})=0$ by definition.
    
    We claim that $L_1(\vec{x})$ is equal to the number of edges on the left-hand side of $e_{1,x_1}$ that point at $v_1$ and intersect $C^*$. We refer to Figure \ref{fig:leftmove} for a specific example. By Lemma \ref{lemma:consecutive1}, we know the edges pointing at $v_1$ that intersect $C^*$ are consecutive and $e_{1,x_1}^*\subset C^*$. Hence there is a leftmost edge, denoted $e_{1,l}$, so that all edges between $e_{1,x_1}$ and $e_{1,l}$ intersects $C^*_x$.  Since $C^*_x$ goes in counter-clockwise direction, all the crossings generated by the edges intersecting $C^*_x$ are assigned to their east corners, which allows a sequence of clockwise local moves from $e_{1,x_1}$ all the way to $e_{1,l}$. After that, the crossing generated by $e_{1, l-1}$ (if exists) was assigned to its west corners, so no more local moves on $v_1$ can be further applied.
    
    Similarly, the fact that $C^*_y$ goes in clockwise direction can be used to show that $R_1(\vec{y})$ is equal to the number of edges on the right-hand side of $e_{1,y_1}=e_{1,x_1+1}$ that intersect $C^*$. This gives 
    \begin{equation*}
    \begin{aligned}
        (R_1(\vec{y})-L_1(\vec{y}))-(R_1(\vec{x})-L_1(\vec{x}))&=(R_1(\vec{y})-0)-(0-L_1(\vec{x}))\\
        &= R_1(\vec{y})+L_1(\vec{x})\\
        &= |\mathcal{E}_1|-2.
    \end{aligned}
    \end{equation*}
Here, we subtract $2$ from $|\mathcal{E}_1|$ because $e_{1,x_1}, e_{1,y_1} \in \mathcal{E}_1$ are the only edges pointing at $v_1$ that are not counted in either $L_1(\vec{x})$ or $R_1(\vec{y})$.

   \begin{figure}[!h]
    \centering
    \includegraphics[width=15cm]{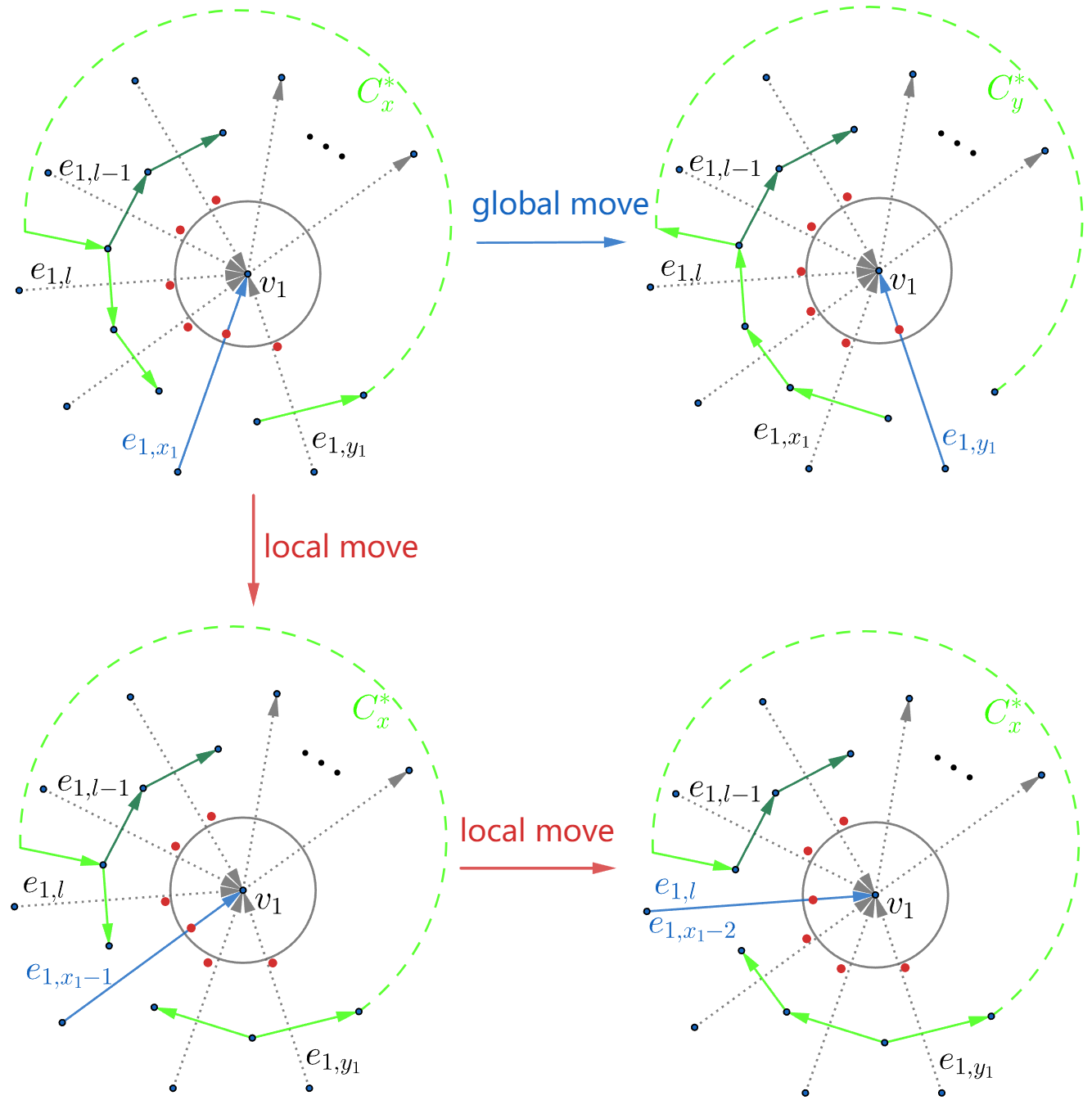}
    \caption{The edges pointing at $v_1$ that intersect $C^*_x$ on the left-hand side of $e_{1,x_1}$ are adjacent to $e_{1,x_1}$, which allow a sequence of local moves (the red arrows) from $e_{1,x_1}$ to them. In this specific example, $L_1(\vec{x})=2$.}
    \label{fig:leftmove}
    \end{figure}

    \bigskip\noindent
    Second, we apply the same idea to determine the relations between $R_i(\vec{x})$, $L_i(\vec{x})$ and $R_i(\vec{y})$, $L_i(\vec{y})$ for all other vertices $v_i$.  There are two cases depending on whether $v_i$ is inside or outside $C^*$. 
    
    Case 1: Assume $v_i$ is inside $C^*$.  Applying Lemma \ref{lemma:orientations} with its notations, we have $L_i(\vec{x})=n-p+1$ since $e_p^*,\cdots,e_n^*$ assign the crossings to their east corners while $e_1^*,\cdots,e_{p-1}^*$ assign the  crossings to their west corners. When we apply the analogous statement of Lemma \ref{lemma:orientations} on $T_{\vec{y}}$ instead, we find $e_p^*,\cdots,e_q^*$ will reverse orientation while $e_1^*,\cdots,e_{p-1}^*$ and $e_{q+1},\cdots,e_n$ keep their original orientations.  Consequently, only $e_{q+1},\cdots,e_n$ can be reached from $e_i$ by local moves, which implies $L_i(\vec{y})=n-q$. Thus, we have 
    $$L_i(\vec{x})-L_i(\vec{y})=q-p+1,$$ which is exactly the number of edges in $\mathcal{E}_i$ on the left-hand side of $e$ (See Figure \ref{fig:leftmove1} for an illustration).
    An identical argument can be applied to show that $R_i(\vec{y})-R_i(\vec{x})$ is equal the number of edges in $\mathcal{E}_i$ on the right-hand side of $e$. Putting together, we have

    \begin{equation*}
    \begin{aligned}
        (R_i(\vec{y})-L_i(\vec{y}))-(R_i(\vec{x})-L_i(\vec{x}))
        &= (R_i(\vec{y})-R_i(\vec{x}))+(L_i(\vec{x})-L_i(\vec{y}))\\
        &= |\mathcal{E}_i|.
    \end{aligned}
    \end{equation*}

   \begin{figure}[!h]
    \centering
    \includegraphics[width=16cm]{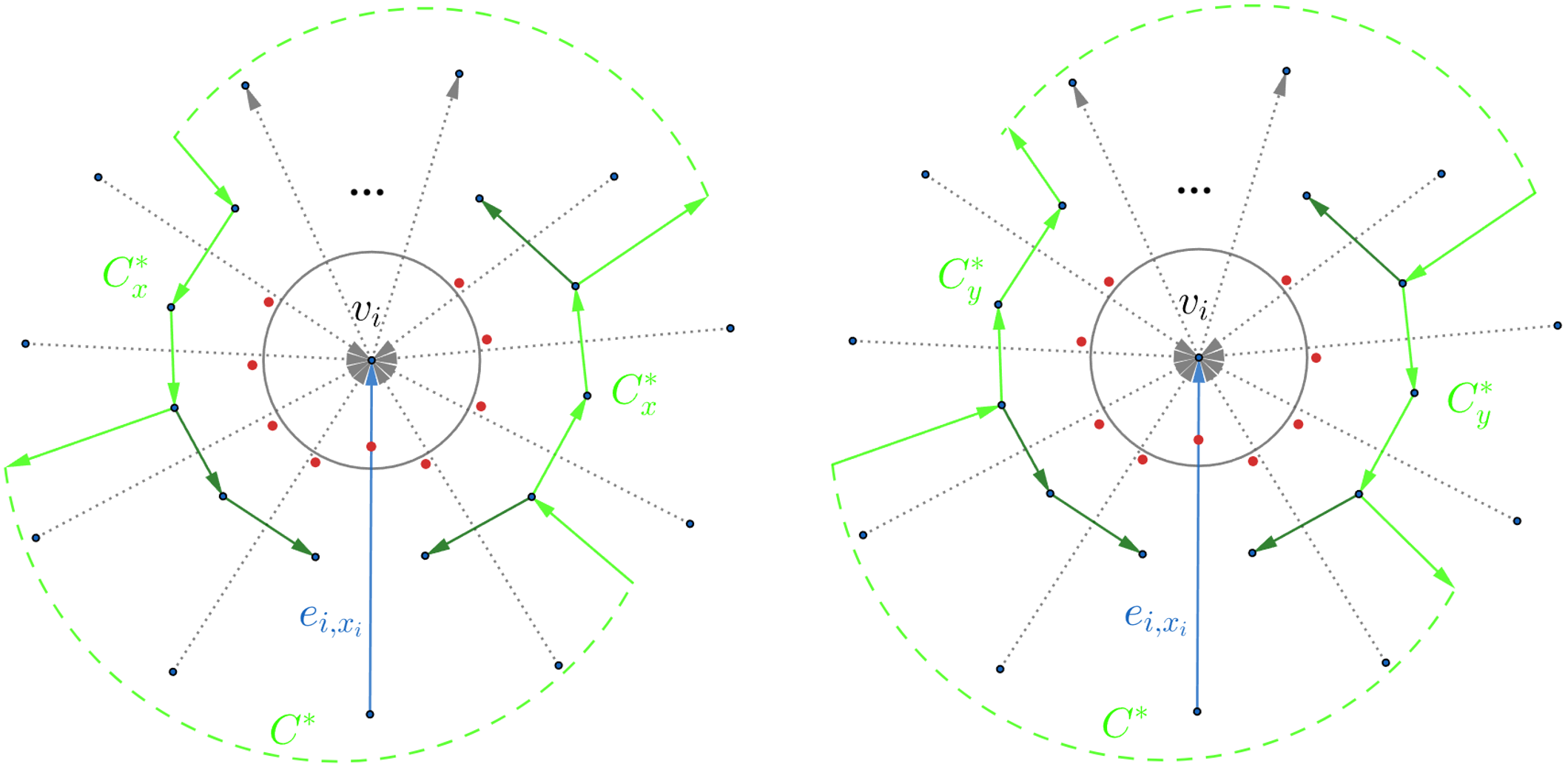}
    \caption{Left: All crossing on the left-hand side of $e_i$ are assigned to their east corners, thus $L_i(\vec{x})=4$. Only one crossing on the right-hand side of $e_i$ is assigned to its west corner, and this edge is next to $e_i$, thus $R_i(\vec{x})=1$. Right: In this example, $L_i(\vec{y})=2$ and $R_i(\vec{y})=3$. So $L_i(\vec{x})-L_i(\vec{y})=2$ equals the number of edges on the left-hand side of $e_i$ that intersect $C^*$ and $R_i(\vec{y})-R_i(\vec{x})=2$ equals the number of edges on the right-hand side of $e_i$ that intersect $C^*$.}
    \label{fig:leftmove1}
\end{figure}

\medskip 
Case 2: Assume $v_i$ is outside $C^*$.  Applying the analogous statement of Lemma \ref{lemma:orientations} on an edge ``pointing out of'' $C^*$ instead, we can show that $L_i(\vec{y})-L_i(\vec{x})$ is equal to the number of edges in $\mathcal{E}_i$ which lie on the left-hand side of $e_i$, and $R_i(\vec{x})-R_i(\vec{y})$ is equal to the number of edges in $\mathcal{E}_i$ which lie on the right-hand side of $e_i$. Thus
    \begin{equation*}
    \begin{aligned}
        (R_i(\vec{y})-L_i(\vec{y}))-(R_i(\vec{x})-L_i(\vec{x}))
        &= (R_i(\vec{y})-R_i(\vec{x}))+(L_i(\vec{x})-L_i(\vec{y}))\\
        &= -|\mathcal{E}_i|.
    \end{aligned}
    \end{equation*}


\bigskip
    
Last, recall that in the proof of Lemma \ref{Lemma:degreedifference} we showed that the number of edges pointing into and out of $C^*$ is equal. In other words, $$\sum_{v_i\text{ inside } C^*}|\mathcal{E}_i|=\sum_{v_j\text{ outside } C^*}|\mathcal{E}_j|.$$

Now we can use (\ref{eq:average}) to compute

\begin{equation*}
    \begin{aligned}
        A(X_\beta)
        = &\sum_{p=1}^k (y_p+\frac{1}{2}(R_p(\vec{y})-L_p(\vec{y})))\\
        = &1+\sum_{p=1}^k x_p + \sum_{v_i\text{ inside } C^*}\frac{1}{2}(R_i(\vec{x})-L_i(\vec{x})+|\mathcal{E}_i|)\\
         &+\frac{1}{2}(-2) + \sum_{v_j\text{ outside } C^*}\frac{1}{2}(R_j(\vec{x})-L_j(\vec{x})-|\mathcal{E}_j|)\\
         = &\sum_{p=1}^k (x_p + \frac{1}{2}(R_p(\vec{x})-L_p(\vec{x})))\\
         &+\frac{1}{2}\left(\sum_{v_i\text{ inside } C^*}|\mathcal{E}_i|-\sum_{v_j\text{ outside } C^*}|\mathcal{E}_j|\right)\\
         = &\sum_{p=1}^k (x_p + \frac{1}{2}(R_p(\vec{x})-L_p(\vec{x})))\\
         = & A(X_\alpha).
    \end{aligned}
    \end{equation*}
    
    This completes the proof that all maximal rectangles share the same symmetry axis.

\end{proof}

\begin{lemma}\label{lemma:cauchyproduct}
    If $(a_i)_{i=1}^n$ and $(b_j)_{j=1}^m$ are sequences satisfying the trapezoidal property, then the Cauchy product of $(a_i)_{i=1}^n$ and $(b_j)_{j=1}^m$ i.e., the coefficients of 
    $$p(t)=(\sum_{i=1}^n a_i t^i)(\sum_{j=1}^m b_j t^j)$$ also form a trapezoidal sequence.
\end{lemma}

\begin{proof}
     Let $$\phi_p=\sum_{k=-p}^p t^k,\; \psi_p=\sum_{k=-p}^{p+1}t^{k-\frac{1}{2}}.$$ Observe that a sequence $(a_i)_{i\in I}$ satisfies the trapezoidal property if and only if
     $$\sum_{i\in I}a_it^i\doteq\sum_{p=g}^hc_p\phi_p\ \ \ \text{or}\ \ \ \sum_{i\in I}a_it^i\doteq\sum_{p=g}^hc_p\psi_p$$
for some $g,h\in\mathbb{N}$ and some $c_p\in\mathbb{R}_+$.

Then $(\sum_{i=1}^n a_i t^i)(\sum_{j=1}^m b_j t^j)$ is equal to (up to a power of $t$) a linear combination of $\phi_p\phi_q$'s, or $\phi_p\psi_q$'s, or $\psi_p\psi_q$'s. Each product is symmetric around degree $0$. If every product satisfies the trapezoidal property, then their sum clearly satisfies the trapezoidal property. Indeed, we can compute them directly:
$$\phi_p\phi_q=\sum_{k=|p-q|}^{p+q}\phi_k,\ \ \ \phi_p\psi_q=\sum_{k=|p-q|}^{p+q}\psi_k,\ \ \ \psi_q\psi_q=\sum_{k=|p-q|}^{p+q+1}\phi_k,$$
all of which satisfies the trapezoidal property.

\end{proof}

\begin{proof}[Proof of Theorem \ref{Thm:Trapezoidal}]

    By Proposition \ref{proposition:rectangledecomposition} and Proposition \ref{proposition:sameaveragenorm}, the sublattice representing spanning trees can be decomposed into a disjoint union of maximal rectangles with the same symmetry axis. Each maximal rectangle contributes 
    $$\Delta_{X_\alpha}(t)=\prod_{i=1}^k(\sum_{j=m_{\alpha,i}}^{M_{\alpha, i}} t^j)$$
    in the Alexander polynomial $\Delta_G(t)$. That the coefficients of $\sum_{j=m_{\alpha,i}}^{M_{\alpha, i}} t^j$ trivially satisfies the trapezoidal property implies their Cauchy product $\Delta_{X_\alpha}(t)$ also satisfies the trapezoidal property by Lemma \ref{lemma:cauchyproduct}. Since all these symmetric polynomials $\Delta_{X_\alpha}(t)$ have the same middle degree, their sum $\Delta_G(t)=\sum_\alpha \Delta_{X_\alpha}(t)$ will have unimodal coefficients.



\end{proof}

\section{Fox's trapezoidal conjecture for alternating knots}

We conclude this paper with a discussion on comparison and an explicit connection between our trapezoidal theorem on graphs and the following classical trapezoidal conjecture on alternating knots \cite{Fox}. 

\begin{conjecture}[Fox's conjecture]
Let $K$ be an alternating knot.  Then the coefficients of $\Delta_K(-t)$ form a unimodal sequence. 
\end{conjecture}

In \cite{Cro59}, Crowell exhibits a spanning tree model for the Alexander polynomial of alternating knots as follows: starting from the alternating diagram $K$, one can define an oriented plane graph $\mathcal{G}(K)$ by letting
\begin{enumerate}
    \item[(i)] The vertices of $\mathcal{G}(K)$ are the crossings of $K$.  
    \item[(ii)] The edges of $\mathcal{G}(K)$ are the arcs between the crossings, and are oriented from the overcrossings to the undercrossings.  
\end{enumerate}

\begin{figure}[!h]
    \centering
    \includegraphics[width=8cm]{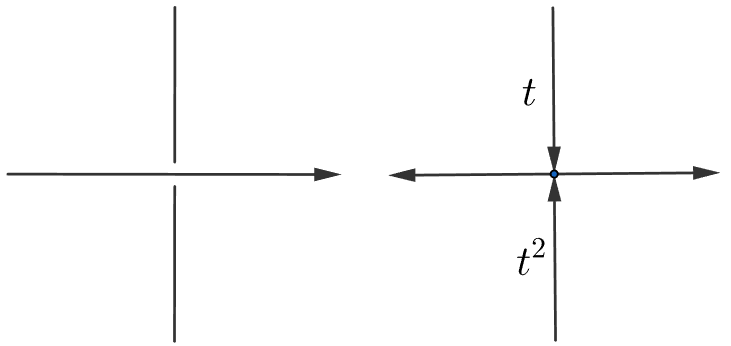}
    \caption{Left: An arbitrary orientation of $K$. Right: The two over-crossing (resp. under-crossing) arcs at each crossing correspond to outgoing (resp. outgoing) edges. Furthermore, we label the edges with weights $t^2$ and $t$ such that, as one traverses the over-strand at crossing, the weight $t^2$ (resp. $t$) appears on the right (resp. left).}
    \label{fig:crowell}
\end{figure}

\begin{theorem} [Crowell] Assign a weight on each edge of the induced graph $\mathcal{G}(K)$ by the rule in Figure \ref{fig:crowell}. If we fix an arbitrary root $r\in \mathcal{G}(K)$ and let $\mathcal{T}_r(\mathcal{G}(K))$ denote the set of rooted spanning trees, then $$\Delta_K(-t)\doteq \sum _{T\in \mathcal{T}_r(\mathcal{G}(K))} \prod_{e\in T} w(e).   $$

\end{theorem}

In order to gain further insight, it would be helpful to make an explicit comparison with our spanning tree model in Theorem \ref{Thm:spanningtreemodel}. Upon making the identification between $X$ and $\mathcal{T}_r(G)$ via Definition \ref{Def:lattice} and assigning the weight $w(e_{i, x_i})=t^{x_i}$ on the $x_i$-th edge entering the vertex $v_i$, we may rewrite Formula (\ref{eq:spanning}) as 
\begin{equation}
\Delta_G(t)\doteq \sum _{T\in \mathcal{T}_r(G)}\prod_{e\in T} w(e),
\end{equation} 
which exhibits a strong formal similarity to the above formula by Crowell.  However, it is important to note that $\mathcal{G}(K)$ is not an MOY graph, as none of the orientation is transverse at any of its vertices. Accordingly, Theorem \ref{Thm:Trapezoidal} does not apply to $\mathcal{G}(K)$, thus leaving Fox' conjecture unresolved. 

\begin{example}\label{exp:trefoil}
Let $K_1$ be the trefoil knot, and $\mathcal{G}(K_1)$ the induced Crowell graph. If we compare $\mathcal{G}(K_1)$ with the planar singular trefoil $\mathcal{K}_1$ side-by-side (Figure \ref{fig:trefoil}), we find that they have the same vertex and edge sets but different orientations. 

\begin{figure}[!h]
    \centering
    \includegraphics[width=16cm]{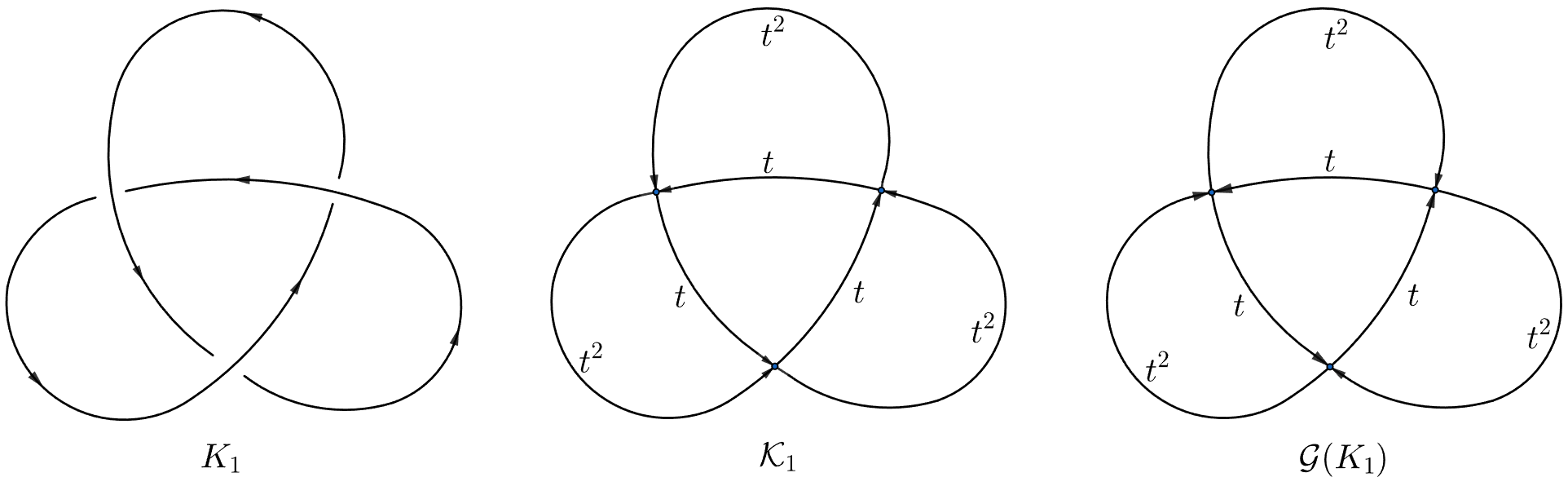}
    \caption{Left: An oriented trefoil knot $K_1$. Middle: An oriented singular trefoil knot $\mathcal{K}_1$. Right: The Crowell graph $\mathcal{G}(K_1)$.}
    \label{fig:trefoil}
\end{figure}

When we fix one of the vertices, say $v_1$, as the root $r$ and compare their rooted spanning trees, it is interesting to note that $|\mathcal{T}_r(\mathcal{G}(K_1))|=3$ and $|\mathcal{T}_r(\mathcal{K}_1)|=4$.  See Figure \ref{fig:spanningtree1} and \ref{fig:spanningtree2}, respectively.  This also gives rise to distinct Alexander polynomials, namely, $$\Delta_{K_1}(-t)\doteq 1+t+t^2;\, \Delta_{\mathcal{K}_1}(t)\doteq 1+2t+t^2.$$

\begin{figure}[!h]
    \centering
    \includegraphics[width=12cm]{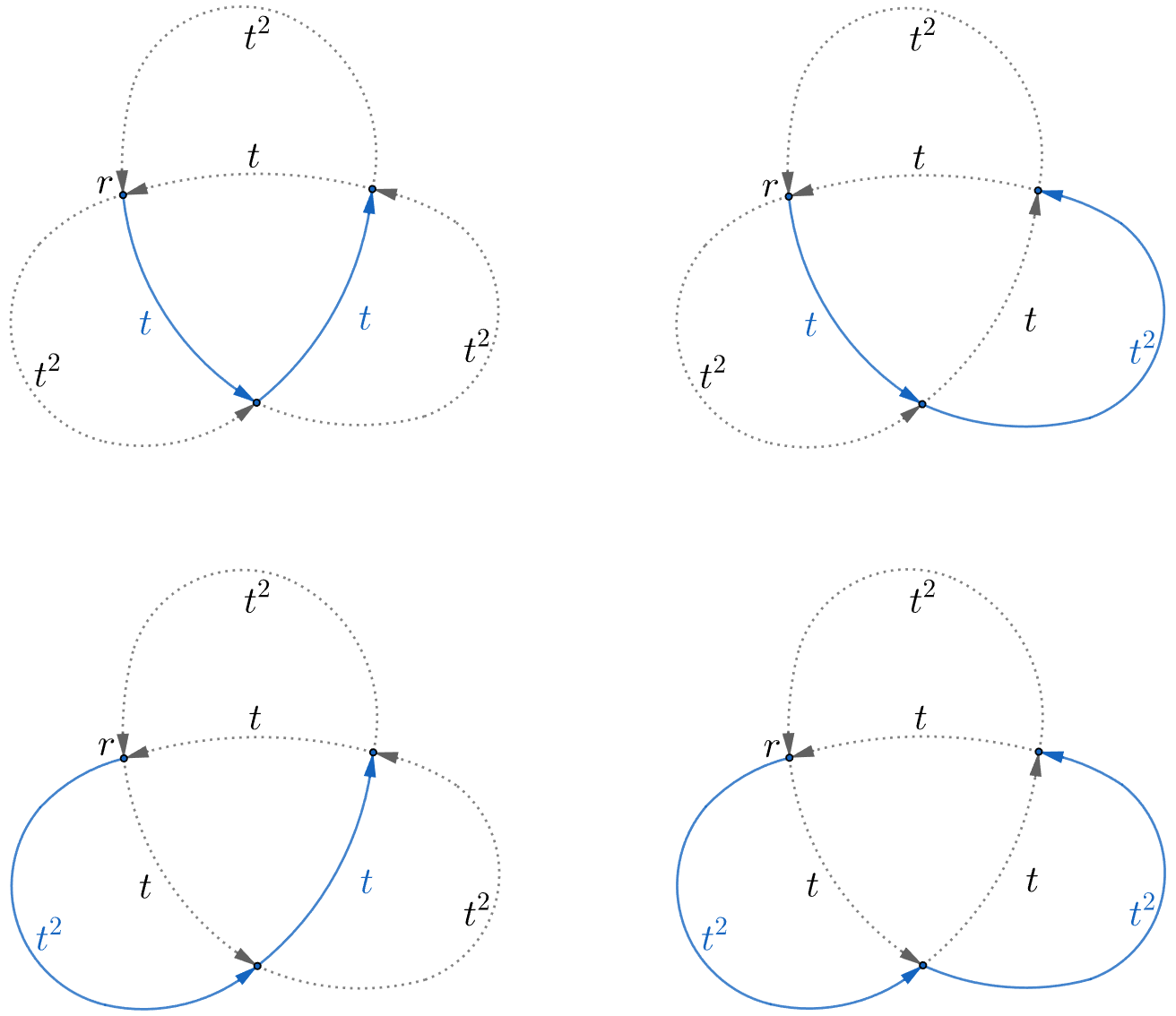}
    \caption{The four spanning trees of $\mathcal{K}_1$. They respectively contribute $t^2$, $t^3$, $t^3$ and $t^4$. Thus $\Delta_{\mathcal{K}_1}(t)\doteq t^2+2t^3+t^4\doteq 1+2t+t^2$.}
    \label{fig:spanningtree1}
\end{figure}

\begin{figure}[!h]
    \centering
    \includegraphics[width=15cm]{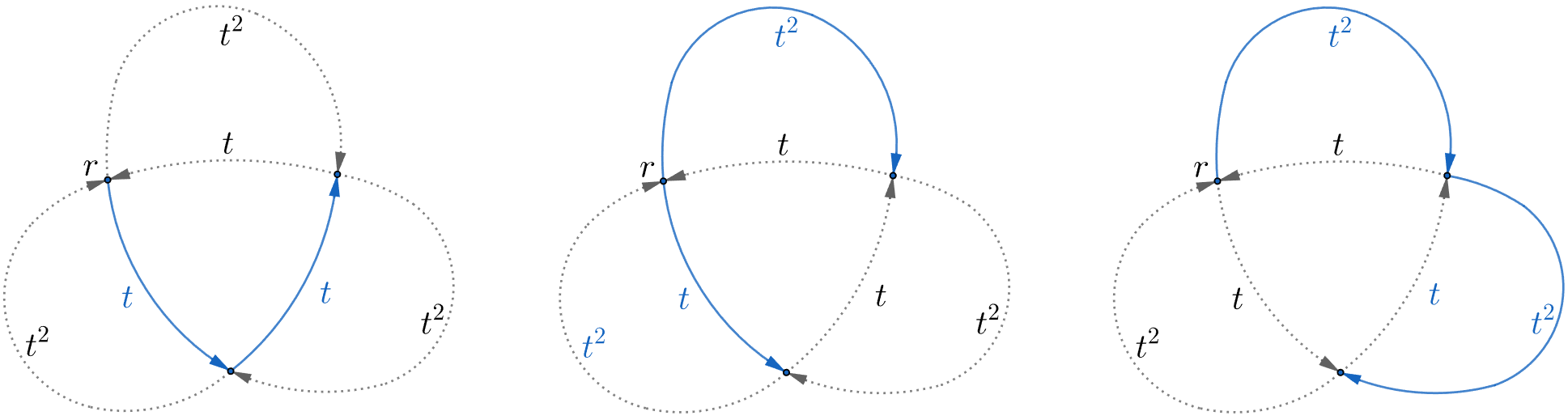}
    \caption{The three spanning trees of $\mathcal{G}(K_1)$. They respectively contribute $t^2$, $t^3$ and $t^4$. Thus $\Delta_{K_1}(-t)\doteq t^2+t^3+t^4\doteq 1+t+t^2$.}
    \label{fig:spanningtree2}
\end{figure}

\end{example}

\begin{example}\label{exp:figure8}

Let $K_2$ be the figure-eight knot and $\mathcal{G}(K_2)$ the induced Crowell graph. In this particular case, we have the inspiring discovery that $\mathcal{G}(K_2)$ and $\mathcal{K}_2$ are, in fact, isomorphic even when taking into account the weights assigned to the edges (Figure \ref{fig:figureeight}).  While their plane embeddings are different, this does not affect the isomorphism between their rooted spanning trees $\mathcal{T}_r(\mathcal{G}(K_2))$ and $\mathcal{T}_r(\mathcal{K}_2)$.  This explains the identity 
$$\Delta_{K_2}(-t)=\Delta_{\mathcal{K}_2}(t)\doteq 1+3t+t^2.$$

\begin{figure}[!h]
    \centering
    \includegraphics[width=15cm]{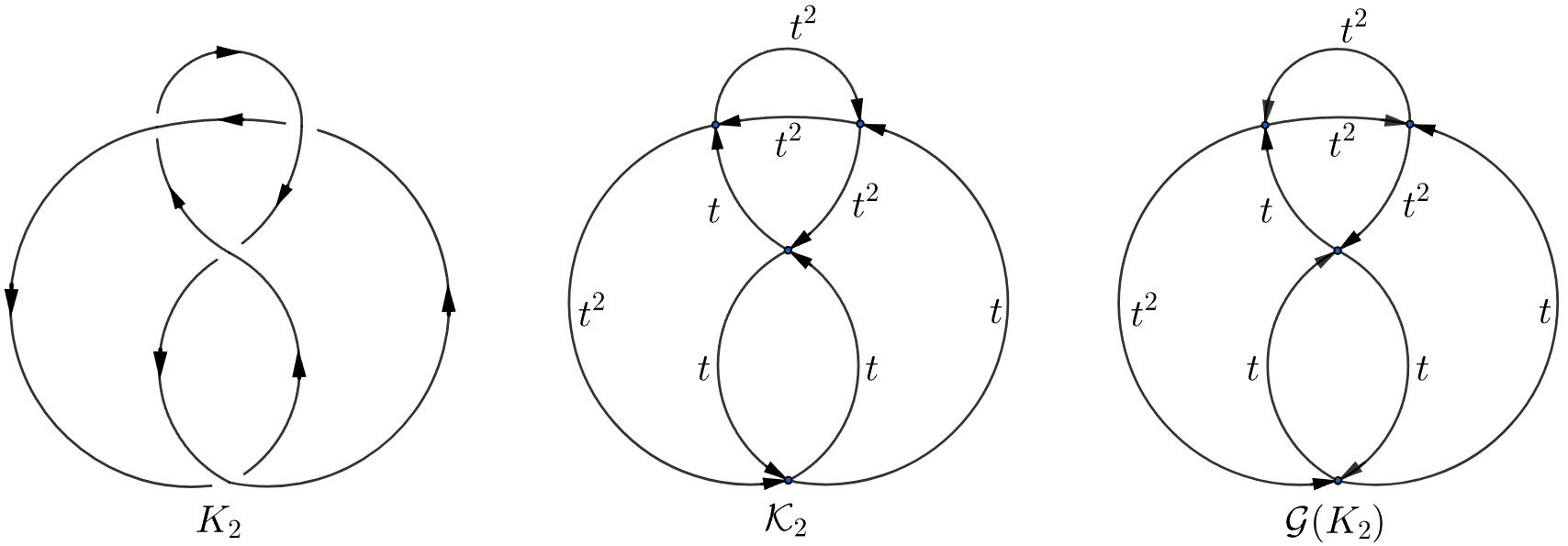}
    \caption{Left: An oriented figure-eight knot $K_2$. Middle: An oriented singular figure-eight knot $\mathcal{K}_2$. Right: The Crowell graph $\mathcal{G}(K_2)$.}
    \label{fig:figureeight}
\end{figure}

\end{example}

In light of the above examples, we pose the following question:

\bigskip

\noindent
\textbf{\textit{Question}}: Given an alternating knot $K$, is there a natural construction of a plane MOY graph $G$ from $K$ such that
$\Delta_K(-t)\doteq\Delta_G(t)$?

\bigskip


An affirmative answer to this question would resolve Fox's trapezoidal conjecture.  Our preliminary investigations have yielded promising results for several classes of alternating knots, including specific subclasses of positive/negative knots and algebraic knots. While $G$ can be constructed directly as the corresponding planar singular knot $\mathcal{K}$ in some favorable cases, such as the figure-eight knot $K_2$, more sophisticated constructions are often needed for knots like the trefoil $K_1$, which involve taking {\it Tait graphs} of the original knot.  This work is still in progress, and we will not elaborate further here.

\end{document}